\documentclass[11pt,reqno]{amsart}
\usepackage[utf8]{inputenc}
\usepackage[T1]{fontenc}
\usepackage{lmodern}
\usepackage[a4paper, total={6in, 9in},marginparwidth=2.5cm]{geometry}
\usepackage{amsmath}
\usepackage{amssymb}
\usepackage{amsfonts}
\usepackage{mathrsfs}
\usepackage{amsthm}
\usepackage{commath}
\usepackage[colorlinks,citecolor=cyan,pagebackref,hypertexnames=false]{hyperref}
\usepackage{doi}
\usepackage{xcolor}
\usepackage{dsfont}
\usepackage{marginnote}
\usepackage{enumitem}

\newcommand{\R}{\mathbb{R}}
\newcommand{\Z}{\mathbb{Z}}
\newcommand{\N}{\mathbb{N}}

\newcommand\B{{\mathbb{B}}}
\newcommand{\D}{\mathcal{D}}
\newcommand{\cG}{\mathcal{G}}
\newcommand{\cL}{\mathcal{L}}
\newcommand{\cP}{\mathcal{P}}
\newcommand{\sB}{\mathsf{B}}

\newcommand{\Tree}{\mathsf{Tree}}

\newcommand{\Ch}{\mathsf{Ch}}
\newcommand{\one}{\mathds{1}}
\newcommand{\sC}{\mathscr{C}}
\DeclareMathOperator{\supp}{supp}

\DeclareMathOperator{\lip}{Lip}
\DeclareMathOperator{\diam}{diam}
\DeclareMathOperator{\dist}{dist}
\newcommand\restr[2]{{
		\left.\kern-\nulldelimiterspace 
		#1 
		\right|_{#2} 
}}
\newcommand{\mres}{\mathbin{\vrule height 1.6ex depth 0pt width
		0.13ex\vrule height 0.13ex depth 0pt width 1.3ex}}

\newcommand{\Hr}[1]{{\mathcal{H}^d}\mres {#1}}
\newcommand{\Hd}{\mathcal{H}^d}

\newcommand{\avg}[1]{\langle #1\rangle}

\usepackage[disable]{todonotes}

\usepackage{tikz,amsthm,amsmath,amstext,amssymb,amscd,epsfig,euscript, mathrsfs, dsfont,pspicture,multicol,graphpap,graphics,graphicx,times,subfig,sidecap,wrapfig,color,pict2e}

\newcommand{\wt}[1]{{\widetilde{#1}}}

\theoremstyle{plain}
\newtheorem*{maintheorem}{Main Theorem}
\newtheorem{theorem}{Theorem}

\newtheorem{lemma}[theorem]{Lemma}

\newtheorem{prop}[theorem]{Proposition}

\newtheorem*{question*}{Question}

\theoremstyle{definition}

\newtheorem{remark}[theorem]{Remark}
\numberwithin{equation}{section}
\numberwithin{theorem}{section}

\def\lec{\lesssim}
\def\R{\mathbb{R}}
\def\Z{\mathbb{Z}}

\def\ve{\varepsilon}
\newcommand{\ps}[1]{\left( #1 \right)}

\newcommand{\av}[1]{\left| #1 \right|}
\newcommand{\nrm}[1]{\left\lVert #1 \right\rVert}
\def\HH{\mathcal{H}}
\def\wt 
\def\ve{\varepsilon}
\def\Lip{{\rm Lip}}

\makeatletter
\def\@tocline#1#2#3#4#5#6#7{\relax
  \ifnum #1>\c@tocdepth 
  \else
    \par \addpenalty\@secpenalty\addvspace{#2}%
    \begingroup \hyphenpenalty\@M
    \@ifempty{#4}{%
      \@tempdima\csname r@tocindent\number#1\endcsname\relax
    }{%
      \@tempdima#4\relax
    }%
    \parindent\z@ \leftskip#3\relax \advance\leftskip\@tempdima\relax
    \rightskip\@pnumwidth plus4em \parfillskip-\@pnumwidth
    #5\leavevmode\hskip-\@tempdima
      \ifcase #1
       \or\or \hskip 1em \or \hskip 2em \else \hskip 3em \fi%
      #6\nobreak\relax
    \dotfill\hbox to\@pnumwidth{\@tocpagenum{#7}}\par
    \nobreak
    \endgroup
  \fi}
\makeatother

\begin{document}

\title[An $\alpha$-number characterization of $L^{p}$ spaces]{An $\alpha$-number characterization of $L^{p}$ spaces on uniformly rectifiable sets}

\author[Azzam]{Jonas Azzam}

\address{Jonas Azzam\\
School of Mathematics \\ University of Edinburgh \\ JCMB, Kings Buildings \\
Mayfield Road, Edinburgh,
EH9 3JZ, Scotland.}
\email{j.azzam "at" ed.ac.uk}

\author[D\k{a}browski]{Damian D\k{a}browski}
\address{Damian D\k{a}browski\\ Departament de Matem\`atiques\\ Universitat Aut\`onoma de Barcelona; Barcelona Graduate School of Mathematics (BGSMath) \\ Edifici C Facultat de Ci\`encies\\
08193 Bellaterra (Barcelona), Catalonia, Spain }
\email{ddabrowski "at" mat.uab.cat}

\begin{abstract}
	We give a characterization of $L^{p}(\sigma)$ for uniformly rectifiable measures $\sigma$ using Tolsa's $\alpha$-numbers, by showing, for $1<p<\infty$ and $f\in L^{p}(\sigma)$, that
	\[
	\norm{f}_{L^{p}(\sigma)}\sim \nrm{\left(\int_{0}^{\infty} \left(\alpha_{f\sigma}(x,r)+|f|_{x,r}\alpha_{\sigma}(x,r)\right)^2\ \frac{dr}{r} \right)^{\frac{1}{2}}}_{L^{p}(\sigma)}.
	\]
\end{abstract}

\maketitle

\def\loc{{\rm loc}}

\tableofcontents

\section{Introduction}

We say a measure $\mu$ in $\R^{n}$ is {\it $d$-rectifiable} if $\mu$ is absolutely continuous with respect to $d$-dimensional Hausdorff measure and we may exhaust $\mu$-almost all of $\R^{n}$ by countably many Lipschitz graphs. It is a classical result that, for $\mu$-a.e. $x\in \R^{n}$, the densities $\mu(B(x,r))/r^{d}$ stabilize as $r\rightarrow 0$ in the sense that they converge to a nonzero constant and so on small scales the measure $\mu$ scales like $d$-dimensional Lebesgue measure. What's more is that the {\it shape} of the measure $\mu$ also stabilizes: as we zoom in on $x$, if we set $\mu_{x,r}(A)= \mu(rA+x)$, then $\mu_{x,r}r^{-d}$ converges weakly to $d$-dimensional Lebesgue measure on some $d$-dimensional plane. 

In \cite{Tol09}, Tolsa quantified how much a uniformly rectifiable measure can deviate from resembling planar Lebesgue measure. Recall that a measure $\mu$ is {\it uniformly rectifiable} (UR) if firstly, it is {\it Ahlfors $d$-regular}, meaning there is $A>0$ so that 
\[
A^{-1}r^{d}\leq \sigma(B(x,r))\leq Ar^{d} \;\; \mbox{ for all }x\in \supp \sigma,\ 0<r<\diam (\supp \sigma),
\]
and $\sigma$ has {\it big pieces of Lipschitz images} (BPLI), meaning there are constants $L,c>0$ so that for each $x\in \supp \sigma$ and $0<r<\diam (\supp \sigma)$ there is an $L$-Lipschitz mapping $f:B_{d}(0,r)\rightarrow B(x,r)$ so that $\mu(f(B_{d}(0,r))$. We say a set $E\subseteq \R^{n}$ is UR if $\HH^{d}|_{E}$ is UR.

Before stating Tolsa's result, we will describe how he measures the planarity of a measure. First, we define a distance between measures. For two measures $\mu$ and $\nu$ and a ball $B$ we define
$$F_B(\sigma,\nu):= \sup\Bigl\{ \Bigl|{\textstyle \int \phi \,d\sigma  -
	\int \phi\,d\nu}\Bigr|:\, \phi\in \Lip_{1}(B) \Bigr\},$$
where $\Lip_{1}(B)$ is the set of $1$-Lipschitz functions supported in $B$. This is a variant of the Wasserstein $1$-distance from mass transport theory. See \cite[Chapter 14]{mattila1999geometry} for a discussion about this distance. 

For a (possibly real-valued) measure $\mu$ and $d\in \N$, if $B=B(x,r)$ we define

\begin{equation}\label{alpha-def}
\alpha_\mu^{d}(x,r)= \alpha_\mu^{d}(B) := \frac1{r^{d+1}}\,\inf_{c\in\R,L} \,F_{B}(\mu,\,c\HH^{d}|_{L}),
\end{equation}

We will often omit the superscript $d$, as it will be fixed throughout. 	

\begin{theorem}[{\cite[Theorem 1.2]{Tol09}}]
\label{t:xavi}
An Ahlfors $d$-regular measure $\sigma$ is UR if and only if the measure $\alpha_{\sigma}^{d}(x,r)^{2}d\sigma(x)\frac{dr}{r}$ is a Carleson measure, meaning that for all balls $B$ centered on $\supp \sigma$ with $0<r_{B}<\diam (\supp \sigma)$,
\[
\int_{0}^{r_{B}}\int_{B} \alpha_{\sigma}^{d}(x,r)^{2}d\sigma(x)\frac{dr}{r}\leq C\sigma(B)
\]
for some fixed $C>0$.
\end{theorem}

Estimates on $\alpha$-type numbers are particularly useful in studying rectifiability. From a geometric view point, they give quite a lot of information about the shape of a measure. David and Semmes \cite{DS} gave an earlier characterization of UR sets in terms of a Carleson meausure condition on $\beta$-numbers, which are quantities like $\alpha$-numbers except they only measure the average distance of a measure to a plane, so while a measure could be very close to lying on a plane, its mass could be very unevenly distributed resulting in a large $\alpha$-number. The additional information provided by the $\alpha$-numbers was crucial for the main result of \cite{Tol09}, where Tolsa improved on the work in \cite{DS} by expanding the class of Calderon-Zygmund operators on UR sets that were known to be bounded. See also \cite{Tol08} where $\alpha$-numbers are used to characterize rectifiability of sets of finite measure in terms of existence of principal values for the Riesz transform, and \cite{DEM18,feneuil2020absolute,david2020harmonic} where they are used to study higher co-dimensional analogues of harmonic measure. 

The purpose of this note is to extend Tolsa's result to measures that are not Ahlfors regular, but are given by $L^{p}$ functions defined on UR sets. 

Given a Radon measure $\sigma$, $f\in L^1_{loc}(\sigma)$, and a ball $B=B(x,r)$ with $\sigma(B(x,r))>0$ set
\begin{equation*}
f_B = {f}_{x,r} = \frac{\int_B f\ d\sigma}{\sigma(B)}.
\end{equation*}

\begin{maintheorem}
	Let $\sigma$ be a UR measure and $f\in L^{p}(\sigma)$ where $1<p<\infty$. Then
	\begin{equation}\label{eq:main thm}
		\norm{f}_{L^{p}(\sigma)}\sim \nrm{\left(\int_{0}^{\infty} \left(\alpha_{f\sigma}(x,r)+|f|_{x,r}\alpha_{\sigma}(x,r)\right)^2\ \frac{dr}{r} \right)^{\frac{1}{2}}}_{L^{p}(\sigma)},
	\end{equation}
	with the implicit constant depending on $p$ and $\sigma$.
\end{maintheorem}

\subsection*{Sharpness of the result}
An interesting aspect of our result is the presence of two terms that comprise our square function. We don't know whether the result holds for general UR sets without the second term. Neither of the terms bounds the other in the pointwise sense: one could be zero while the other is nonzero. On the other hand, we don't know whether the norm of the square function involving only ${\alpha}_{f\sigma}$ dominates the one involving only $|f|_{x,r}\alpha_{\sigma}$. The reverse inequality is certainly not true, as the latter square function vanishes if $\sigma$ is the Lebesgue measure on $\Sigma=\R^{d}$.
\begin{question*}
	Let $\sigma$ be a UR measure and $f\in L^{p}(\sigma)$ where $1<p<\infty$. Do we have
	\begin{equation}\label{eq:question}
	\norm{f}_{L^{p}(\sigma)}\lesssim \nrm{\left(\int_{0}^{\infty} \alpha_{f\sigma}(x,r)^2\ \frac{dr}{r} \right)^{\frac{1}{2}}}_{L^{p}(\sigma)}?
	\end{equation}
 Equivalently, is it true that
	\begin{equation*}
		\nrm{\left(\int_{0}^{\infty} (|f|_{x,r}\alpha_{\sigma}(x,r))^2\ \frac{dr}{r} \right)^{\frac{1}{2}}}_{L^{p}(\sigma)}\lesssim \nrm{\left(\int_{0}^{\infty} \alpha_{f\sigma}(x,r)^2\ \frac{dr}{r} \right)^{\frac{1}{2}}}_{L^{p}(\sigma)}?
	\end{equation*}
\end{question*}
The answer to the question above is obviously affirmative in the flat case, i.e. $\sigma=\Hd\mres L$ for $L$ a $d$-dimensional plane. It is also positive if $\sigma$ is an Ahlfors $d$-regular measure on a $d$-dimensional plane $L$, i.e. $\sigma=g\Hd\mres L$ for some $g$ satisfying $A^{-1}\le g\le A$. Indeed, let $\tilde{\sigma}=\Hd\mres L$, so that $f\sigma = fg\tilde{\sigma}$. In that case, by the Main Theorem
\begin{equation*}
	\norm{f}_{L^{p}(\sigma)}\sim_A \norm{fg}_{L^{p}(\tilde{\sigma})} \sim_{p} \nrm{\left(\int_{0}^{\infty} \alpha_{fg\tilde{\sigma}}(x,r)^2\ \frac{dr}{r} \right)^{\frac{1}{2}}}_{L^{p}(\tilde{\sigma})} \sim_A \nrm{\left(\int_{0}^{\infty} \alpha_{f{\sigma}}(x,r)^2\ \frac{dr}{r} \right)^{\frac{1}{2}}}_{L^{p}(\sigma)}.
\end{equation*}
Finally, one could show that \eqref{eq:question} is true for ``sufficiently flat'' UR measures $\sigma$. What we mean by this is that if the constant $C$ from \thmref{t:xavi} is sufficiently small, then some variant of Carleson's embedding theorem can be used\footnote{For $p=2$ use e.g. \cite[Theorem 5.8]{tolsa2014analytic}, for $p\neq 2$ one can show a corresponding statement by proving an appropriate good-lambda inequality, in the spirit of what we do in Section \ref{sec:Jfp < fp} (but simpler).} to show that
\begin{equation*}
	\nrm{\left(\int_{0}^{\infty} \left(|f|_{x,r}\alpha_{\sigma}(x,r)\right)^2\ \frac{dr}{r} \right)^{\frac{1}{2}}}_{L^{p}(\sigma)}\lesssim_p C \norm{f}_{L^{p}(\sigma)}\ll \norm{f}_{L^{p}(\sigma)}.
\end{equation*}
This means that the second term from the square function in \eqref{eq:main thm} can essentially be absorbed by the left hand side. To make this more rigorous, one should perhaps track the dependence of the implicit constants in \eqref{eq:main thm} on the UR constants of $\sigma$ with more diligence than we did. However, the implicit constants can only get \emph{better} as $\sigma$ becomes flatter, and they certainly cannot blow-up as the Carleson constant $C$ goes to $0$: if $\sigma$ satisfies the Carleson condition of \thmref{t:xavi} with some $C$, then it also satisfies it with constant $C'$ for every $C'\ge C$.

\subsection*{Related work} While our focus has been in the Ahlfors regular setting, $\alpha$-numbers have also been used to study measures in more general general settings. In \cite{ATT18}, it was shown that pointwise doubling measures $ \mu$ were $d$-rectifiable on the set where the square function $\int_{0}^{\infty}\alpha_{\mu}(x,r)^2\frac{dr}{r}$ was finite, resolving a question left open in \cite{ADT16}. This paper also exposed some limitations with working with $\alpha$-numbers, as a counterexample showed that the same result is not true for general measures. However, the second author of this paper obtained such a generalization in \cite{Dab19} using a different $\alpha$-number, which measures distance between a measure and planar measure using the Wasserstein $2$-metric, which Tolsa had earlier used to give a characterization of UR measures in \cite{Tol12}. So while using the Wasserstein $2$-distance allows one to get a more complete picture, the $\alpha$-number in \thmref{t:xavi} has a more transparent definition and thus is easier to work with.

In \cite{Orp18} Orponen uses a similar square function to characterize when two measures on the real line (one being doubling) are absolutely continuous, however among a few of the differences between the $\alpha$-numbers he uses and ours, while  we compare distance between a measure and a plane, his numbers compare the distance between the two measures, which is another interesting direction. 

\subsection*{Organization of the article} In Section \ref{sec:preliminaries} we introduce the necessary tools and make some initial reductions. We define also $Jf$, a dyadic variant of the square function from the Main Theorem, see \eqref{Jones alpha function}.

We show that $\nrm{Jf}_2\lec \nrm{f}_2$ in Section \ref{sec:Jf2 < f2}. The proof uses martingale difference operators, and it is inspired by how \thmref{t:xavi} was originally proved, see \cite[Section 4]{Tol09}. In Section \ref{sec:Jfp < fp} we use the estimate $\nrm{Jf}_2\lec \nrm{f}_2$ and an appropriate good-lambda inequality to conclude that $\nrm{Jf}_p\lec \nrm{f}_p$ for general $1<p<\infty$.

Finally, in Section \ref{sec:fp < Jfp} we prove $\nrm{f}_p\lec \nrm{Jf}_p$. To do that we use the Littlewood-Paley theory of David, Journ\'{e} and Semmes \cite{DJS85}.

\subsection*{Acknowledgments}
We would like to thank Xavier Tolsa for helpful discussions about the paper.

Most of work on the article has been done during second author's visit to the University of Edinburgh. We are very grateful to the university staff for their kindness and hospitality.

The second author received support from the Spanish Ministry of Economy and Competitiveness, through the María de Maeztu Programme for Units of Excellence in R\&D (MDM-2014-0445), and also partial support from the Catalan Agency for Management of University and Research Grants (2017-SGR-0395), and from the Spanish Ministry of Science, Innovation and Universities (MTM-2016-77635-P).

\section{Preliminaries}\label{sec:preliminaries}

\subsection{Notation}
	In our estimates we will write $f\lesssim g$ to denote $f\le C g$ for some constant $C$ (the so-called ``implicit constant''). If the implicit constant depends on a parameter $t$, i.e. $C=C(t)$, we will write $f\lec_t g$. The notation $f\sim g$ and $f\sim_t g$ stands for $g\lec f\lec g$ and $g\lec_t f\lec_t f$, respectively. To make the notation lighter, we will usually not track the dependence of $C$ on dimensions $n,\ d,$ on the Ahlfors regularity constant of $\sigma$, or the parameter $1<p<\infty$.
	
	Given $x\in\R^n$ and $r>0$ we denote by $B(x,r)$ the open ball centered at $x$ with radius $r$. Conversely, given a ball $B$ (either open or closed, and either $n$ or $d$-dimensional), $r_B$ and $z_B$ denote the radius and the center of $B$, respectively.
	
	For simplicity, we will sometimes write
	\begin{equation*}
	\norm{f}_p:=\norm{f}_{L^p(\sigma)}.
	\end{equation*}
	
	In the introduction we introduced the notation $f_B$ to signify the average of $f$ over a ball $B$ with respect to $\sigma$. For general Borel sets $E\subset \R^d$ with $\sigma(E)>0$ and $f\in L^1_{loc}(\sigma)$ we will write
	\begin{equation*}
	\avg{f}_E = \frac{\int_E f\ d\sigma}{\sigma(E)}.
	\end{equation*}
	
	For a finite set $I$ we will write $\# I$ to denote the cardinality of $I$.
	
	If $v,w\in\R^n$, then $v\cdot w$ denotes their scalar product.
	
	Given $E, F\subset\R^d,$ $\dist_H(E,F)$ stands for the Hausdorff distance between $E$ and $F$.

	\subsection{Adjacent systems of cubes}\label{sec:adjacent}
	As usual, we will work with a family of subsets of $\supp\sigma=:\Sigma$ that in many ways resemble the family of dyadic cubes on $\R^d$. For this reason we will call these sets ``cubes''. Many different systems of cubes have been constructed throughout the years, beginning with the work of David \cite{david1988morceaux} and Christ \cite{christ1990tb}. In our proof it will be convenient to use adjacent systems of cubes constructed by Hyt\"{o}nen and Tapiola \cite{hytonen2014almost}. One should think of them as a generalization of the translated dyadic grids in $\R^d$, widely used to perform the ``$1/3$ trick''.
	
	First, we will say that a family $\D$ of Borel subsets of $\Sigma$ satisfies \emph{the usual properties of David-Christ cubes} if $\D=\bigcup_{k\in\Z}\D_k$, and for each $k\in\Z$:
	\begin{enumerate}[label=(\alph*)]
		\item for $P,\ Q\in\D_k,\ P\neq Q,$ we have $\sigma(P\cap Q)=\varnothing$,
		\item the sets in $\D_k$ cover $\Sigma$:
		\begin{equation*}
		\Sigma=\bigcup_{Q\in\D_k}Q,
		\end{equation*}
		\item for each $Q\in\D_k$ and each $l\ge k$ 
		\begin{equation*}
		Q = \bigcup_{P\in\D_l : P\subset Q} P,
		\end{equation*}
		\item there exists $0<\delta<1$ (independent of $k$) such that each $Q\in\D_k$ has a center $z_Q\in Q$ satisfying
		\begin{equation}\label{eq:balls and cubes}
		B\big(z_Q,\frac{\delta^k}{5}\big)\cap\Sigma\subset Q\subset B(z_Q,3{\delta^k})\cap\Sigma.
		\end{equation}
		Consequently, as long as $\delta^k\lesssim \diam(\Sigma)$, we have $\sigma(Q)\sim \delta^{kd}$. Set $\ell(Q):=\delta^k$.
		\item the cubes $Q\in\D_k$ have thin boundaries, that is, there exists $\gamma\in (0,1)$ such that for $\eta\in (0,0.1)$ we have
		\begin{equation}\label{eq:thin boundaries}
		\sigma(\{x\in\Sigma\ :\ \dist(x,Q) + \dist(x,\Sigma\setminus Q)<\eta\ell(Q)\})\le \eta^{\gamma}\sigma(Q).
		\end{equation}
	\end{enumerate}
	
	\begin{remark}
		Note that in the above we assume $\D_k$ to be defined for all $k\in\Z$. In the case of unbounded $\Sigma$, this translates to having arbitrarily large cubes as $k\to-\infty$. In the case of compact $\Sigma$, there exists some $k_0$ such that for all $k\le k_0$ we have $\D_k=\{\Sigma\}$. However, in our proof we will assume that $\Sigma$ is unbounded, see \lemref{lem:Sigma unbounded}.
	\end{remark}

	In our setting, the results \cite[Theorem 2.9, Theorem 5.9]{hytonen2014almost} can be summarized as follows.
	\begin{lemma}\label{lem:cubes}
		Let $\sigma$ be a $d$-Ahlfors regular measure on $\R^n$. Then, there exist $1\le N<\infty$ and a small constant $0<\delta<0.01$, depending only on the Ahlfors regularity constants of $\sigma$, such that the following holds. Let $\Omega=\{1,\dots, N\}$. For each $\omega\in \Omega$ we have a system of cubes $\D(\omega)$ satisfying the usual properties of David-Christ cubes, and additionally, for all $x\in\supp\sigma$ and $r>0$ there are $\omega\in\Omega,\ k\in\Z$ and $Q\in\D_k(\omega)$ with
		\begin{equation*}
		B(x,r)\cap\supp\sigma \subset Q
		\end{equation*}
		and 
		\begin{equation*}
		\ell(Q)=\delta^k\sim_{\delta} r.
		\end{equation*}
	\end{lemma}

	\begin{remark}
		The construction in \cite{hytonen2014almost} is valid for general (geometrically) doubling metric spaces, possibly with no underlying measure space structure. The constants $N$ and $\delta$ from \lemref{lem:cubes} depend on the doubling constant of the metric space. Hyt\"{o}nen and Tapiola construct two different kinds of cubes, which they call open and closed cubes, see \cite[Theorem 2.9]{hytonen2014almost}. In the above we consider closed cubes, so that properties (b), (c) and (d) follow immediately from \cite[Theorem 2.9]{hytonen2014almost}. To get the property (a) one uses the fact that interiors of $P$ and $Q$ are disjoint by \cite[(2.11)]{hytonen2014almost}, and then $\sigma(\partial P) = \sigma(\partial Q)=0$ follows from (e). To prove the thin boundaries property (e) one may adapt the proof of Christ \cite[pp. 610--612]{christ1990tb} together with Ahlfors regularity of $\sigma$. We omit the details.
	\end{remark}
	
	From now on, let us fix a uniformly rectifiable measure $\sigma$, with $\Sigma=\supp\sigma$. Let $\Omega,\ \delta$ and $\D(\omega)$ be as in \lemref{lem:cubes}. For simplicity, in our estimates we will not track the dependence of implicit constants on $\delta$. 
	
	For all $\omega\in\Omega$ and $Q\in\D_k(\omega)$ we will write
	\begin{align*}
	\D(Q) &:= \{P\in\D(\omega)\ :\ P\subset Q \},\\
	\Ch(Q) &:= \D(Q)\cap \D_{k+1}(\omega).
	\end{align*}
	The elements of $\Ch(Q)$ will be called children of $Q$, and $Q$ will be called their parent. 
	
	Set
	\begin{equation*}
	B_Q:= B(z_Q,4\ell(Q)),
	\end{equation*}
	so that $Q\subset B_Q\cap\Sigma$, and whenever $P\in\D(Q)$ we also have $B_P\subset B_Q$.
	
	Fix some $\omega_0\in\Omega$, and set
	\begin{equation*}
	\D := \D(\omega_0).
	\end{equation*}
	This will be our system of reference. Given $Q\in\D$ we define $\omega(Q)\in\Omega$ to be the index such that there exists $R(Q)\in\D(\omega(Q))$ satisfying $B_Q\cap\supp\sigma\subset R(Q)$ and $\ell(R(Q))\sim\ell(Q)$. If there is more than one such $\omega$, we simply choose one. We define also $\cG(\omega)\subset\D$ as the family of cubes $Q\in\D$ such that $\omega(Q)=\omega$. Clearly,
	\begin{equation*}
	\bigcup_{\omega\in\Omega}\cG(\omega)=\D.
	\end{equation*}

\subsection{\texorpdfstring{$\alpha$}{alpha}-numbers}
In proving the main theorem, it will be more convenient to work with dyadic versions of the $\alpha$-numbers. Below we will introduce the notation needed for this framework.
	Given a Radon measure $\mu$ we denote by $L_{x,r}^{\mu}$ a minimizing $d$-plane for $\alpha_{\mu}(x,r)$, and by $c_{x,r}^{\mu}$ the corresponding constant. They may be non-unique, in which case we just choose one of the minimizers. Set $\cP_{x,r}^{\mu}=\Hr{L_{x,r}^{\mu}}$ and $\cL_{x,r}^{\mu} = c_{x,r}^{\mu}\cP_{x,r}^{\mu}$. If $B=B(x,r)$ we will also write $L_{B}^{\mu},\ c_B^{\mu}$ etc.
	
	For $Q\in\D$ and a Radon measure $\mu$ we set
	\begin{equation*}
	\alpha_{\mu}(Q) := \alpha_{\mu}({B_Q}).
	\end{equation*}
	We will write $L_Q^{\mu}:=L_{B_Q}^{\mu},\ c_Q^{\mu}:=c_{B_Q}^{\mu}$ etc.
	
{	Observe that whenever $B_1\subset B_2$ are balls, we have $\lip_1(B_1)\subset\lip_1(B_2)$, and so if $r_{B_1}\ge C r_{B_1}$, then
	\begin{multline}\label{eq:alphas on two balls}
	\alpha_{\mu}(B_1)= \frac1{r_{B_1}^{d+1}}\,\inf_{c\in\R,L} \,F_{B_1}(\mu,\,c{\HH^{d}}\mres{L})\\
	\le \frac1{r_{B_1}^{d+1}} F_{B_1}(\mu,\cL_{B_2}^{\mu})\le \frac1{r_{B_1}^{d+1}} F_{B_2}(\mu,\cL_{B_2}^{\mu})\sim_C\alpha_{\mu}(B_2).
	\end{multline}}
	
	Consider the following square function:
	\begin{equation}\label{Jones alpha function}
	J(x) = \bigg( \sum_{x\in Q\in\D} \alpha_{f\sigma}(Q)^2 + |f|_{B_Q}^2\alpha_{\sigma}(Q)^2\bigg)^{1/2}.
	\end{equation}
	
The Main Theorem will follow from the following dyadic version:

\begin{theorem}\label{thm:dyadic main theorem}
Let $\sigma$ be a uniformly rectifiable measure with unbounded support, and let $f\in L^p(\sigma)$ for some $1<p<\infty$. Then
\begin{equation*}
		\norm{Jf}_{L^p(\sigma)} \sim \norm{f}_{L^p(\sigma)}.
		\end{equation*}
\end{theorem}

First, let us show why we may assume that $\supp\sigma$ is unbounded.
\begin{lemma}\label{lem:Sigma unbounded}
	It suffices to only prove Main Theorem in the case that $\supp \sigma$ is unbounded. 
\end{lemma}

\begin{proof}
	Suppose $\sigma$ did have compact support. Without loss of generality, we may assume $\diam (\supp\sigma)=1,\ \supp \sigma \subseteq \B=B(0,1)$, and $L_{\B}^{\sigma}=\R^{d}$. Let
	\[
	\mu = \sigma + \mathcal{P}_{\B}^{\sigma}\mres({\R^{d}\backslash 4\B}).
	\]
	It is not hard to show that $\mu$ is also UR. If Main Theorem holds for UR measures of unbounded support, then it holds for $\mu$.
	Let $f\in L^{p}(\sigma)\subseteq L^{p}(\mu)$ and let 
	\[
	{\theta^{f}_\sigma}(x,r):=\alpha_{f\sigma}(x,r) + |f|_{x,r}\alpha_{\sigma}(x,r),
	\]
	so that, by the Main Theorem,
		\begin{equation}\label{eq:main theorem applied to mu}
		\norm{\bigg(\int_0^{\infty}\theta^f_{\mu}(x,r)^2\ \frac{dr}{r}\bigg)^{1/2}}_{L^p(\mu)}\sim\norm{f}_{L^p(\mu)} = \norm{f}_{L^p(\sigma)}.
		\end{equation}
	Observe that
	\[
	\theta_{\sigma}^{f}(x,r) = \theta_{\mu}^{f}(x,r)  \mbox{  for $x\in \supp \sigma$ and $0<r<2$}.
	\]
	Thus,
	\begin{multline*}
	\nrm{\bigg(\int_{0}^{2}\theta^{f}_{\sigma}(x,r)^2\ \frac{dr}{r}\bigg)^{1/2}}_{L^{p}(\sigma)}
	=\nrm{\bigg(\int_{0}^{2}\theta^{f}_{\mu}(x,r)^2\ \frac{dr}{r}\bigg)^{1/2}}_{L^{p}(\sigma)}\\
	\leq  \nrm{\bigg(\int_{0}^{\infty}\theta^{f}_{\mu}(x,r)^2\ \frac{dr}{r}\bigg)^{1/2}}_{L^{p}(\mu)} 
	\overset{\eqref{eq:main theorem applied to mu}}{\lec} \norm{f}_{L^p(\sigma)}.
	\end{multline*}
	Furthermore, we claim that for any $x\in \supp\sigma$ and $r>2$ we have 
		\begin{equation}\label{eq:theta sigma for large r}
		\theta^{f}_\sigma(x,r)\lesssim r^{-d}|f|_{\B}.
		\end{equation}
		Indeed, since $\supp f\subset\supp\sigma\subset\B$,
		\begin{equation}\label{eq:alpha fsigma for large r}
		\alpha_{f\sigma}(x,r)\le \frac{1}{r^{d+1}} F_{B(x,r)}(f\sigma,0)\le \frac{1}{r^d}\int_{\B} |f|\ d\sigma \sim \frac{1}{r^d}|f|_{\B},
		\end{equation}
		and also
		\begin{equation*}
		|f|_{x,r}\alpha_{\sigma}(x,r) \lesssim\frac{1}{r^d} \int_{\B} |f|\ d\sigma \sim \frac{1}{r^d}|f|_{\B}.
		\end{equation*}
		It follows from \eqref{eq:theta sigma for large r} that
		\[
		\int_{2}^{\infty} \theta^{f}_{\sigma}(x,r)^2\frac{dr}{r}\lesssim \int_{2}^{\infty}|f|_{\B}^2\frac{dr}{r^{2d+1}}\lesssim |f|_{\B}^2,
		\]
	and so 
	\[
	\nrm{\bigg(\int_{0}^{\infty}\theta^{f}_{\sigma}(x,r)^2\ \frac{dr}{r}\bigg)^{1/2}}_{L^{p}(\sigma)}
	\lesssim  \nrm{f}_{L^p(\sigma)}+ \int_{\B}|f|d\sigma\lesssim  \nrm{f}_{L^p(\sigma)} .
	\]
	
	To finish the proof we now need to show the reverse inequality. Notice that since $f$ is supported on $\supp \sigma$, $\alpha_{f\mu}(x,r)=\alpha_{f\sigma}(x,r)$ for all $x\in \supp\sigma$ and $r>0$. We can argue just as in \eqref{eq:alpha fsigma for large r} to get that for $x\in\supp\mu$ and $r\geq 2$,  
		\[
		\alpha_{f\mu}(x,r)=\alpha_{f\sigma}(x,r)
		\leq \frac1{r^{d+1}}F_{B(x,r)}(f\sigma,0)
		\lec  \frac{|f|_{\B}}{r^{d}}
		\lec \frac{|f|_{2\B}}{r^{d}}\alpha_{\sigma}(2\B),
		\]
		where we also used $\alpha_{\sigma}(2\B)\sim 1$.
	
	Hence,
	\begin{align*}
	\int_{\B}& \left(\int_{0}^{\infty}\alpha_{f\mu}(x,r)^{2}\frac{dr}{r}\right)^{\frac{p}{2}}d\mu(x)\\
	&\lec \int_{\B}\left(\int_{0}^{2}\alpha_{f\mu}(x,r)^{2}\frac{dr}{r}\right)^{\frac{p}{2}}d\mu(x)+\int_{\B}\left(\int_{2}^{\infty}\alpha_{f\mu}(x,r)^{2}\frac{dr}{r}\right)^{\frac{p}{2}}d\mu(x)\\
	&\lec \int_{\B}\left(\int_{0}^{2}\alpha_{f\sigma}(x,r)^{2}\frac{dr}{r}\right)^{\frac{p}{2}}d\mu(x)+\int_{\B}\left(\int_{2}^{\infty}(|f|_{2\B}\alpha_{\sigma}(2\B))^2\frac{dr}{r^{2d+1}}\right)^{\frac{p}{2}}d\mu(x)\\
	&\lec \int_{\B}\left(\int_{0}^{2}\alpha_{f\sigma}(x,r)^{2}\frac{dr}{r}\right)^{\frac{p}{2}}d\mu(x)+(|f|_{2\B}\alpha_{\sigma}(2\B))^{p}
	\lec \int_{\B}\left(\int_{0}^{\infty}\theta_{\sigma}^f(x,r)^{2}\frac{dr}{r}\right)^{\frac{p}{2}}d\mu(x)
	\end{align*}
	where we used \eqref{eq:alphas on two balls} in the final inequality. 
	
	Furthermore, for $x\in \R^{d}\backslash 4\B$, if $\alpha_{f\mu}(x,r)\neq 0$, then $r\geq |x|/2$ and so
	\begin{align*}
	\int_{\R^{d}\backslash 4\B} \left(\int_{0}^{\infty}\alpha_{f\mu}(x,r)^{2}\frac{dr}{r}\right)^{\frac{p}{2}}d\mu(x)
	&=\sum_{j=2}^{\infty}\int_{\R^{d}\cap (2^{j+1}\B\backslash 2^{j}\B)} \left(\int_{|x|/2}^{\infty}(|f|_{2\B}\alpha_{\sigma}(2\B))^2\frac{dr}{r^{2d+1}}\right)^{\frac{p}{2}}d\mu(x)\\
	& \lec (|f|_{2\B}\alpha_{\sigma}(2\B))^{p} \sum_{j=2}^{\infty}\int_{\R^{d}\cap (2^{j+1}\B\backslash 2^{j}\B)}|x|^{-pd} d\mu(x)
	\\ &\lec (|f|_{2\B}\alpha_{\sigma}(2\B))^{p}
	\lec\int_{\B}\left(\int_{0}^{\infty}|f|_{x,r}\alpha_{\sigma}(x,r)^{2}\frac{dr}{r}\right)^{\frac{p}{2}}d\mu(x)
	\end{align*}
	again using \eqref{eq:alphas on two balls}. These two estimates imply 
	\begin{equation}
	\label{e:alphafmu<alphafsigma}
	\nrm{\bigg(\int_{0}^{\infty}\alpha_{f\mu}(x,r)^{2}\frac{dr}{r}\bigg)^{1/2}}_{L^{p}(\mu)}
	\lec \nrm{\bigg(\int_{0}^{\infty}\theta_{\sigma}^f(x,r)^{2}\frac{dr}{r}\bigg)^{1/2}}_{L^{p}(\sigma)}.
	\end{equation}
	
	Note that for $x\in\supp\sigma$ and $r<2$, we have $\mathcal{P}_{x,r}^{\sigma}=\mathcal{P}_{x,r}^{\mu}$. For $r\geq 2$, notice that $\alpha_{\sigma}(2\B)\sim 1$, and so 
	\[
	\alpha_{\mu}(x,r)
	\leq \frac1{r^{d+1}}F_{B(x,r)}(\mu,\mathcal{P}^{\sigma}_{\B})
	=\frac1{r^{d+1}}F_{B(x,r)}(\sigma,\mathcal{P}^{\sigma}_{\B}\mres{4B})
	\lec r^{-d}\lec \frac{\alpha_{\sigma}(2\B)}{r^{d}},
	\]
	hence 
	\[
	|f|_{x,r}^{\mu}\alpha_{\mu}(x,r)
	\leq |f|_{2\B}\frac{\alpha_{\sigma}(2\B)}{r^{2d}},
	\]
	where $|f|_{x,r}^{\mu} = \int_{B(x,r)} f\ d\mu/\mu(B(x,r))$.
	Thus, just as how we proved \eqref{e:alphafmu<alphafsigma}, we can show
	\[
	\nrm{\int_{0}^{\infty}(|f|^{\mu}_{x,r})^2\alpha_{\mu}(x,r)^{2}\frac{dr}{r}}_{L^{p}(\mu)}
	\lec \nrm{\int_{0}^{\infty}|f|^2_{x,r}\alpha_{\sigma}(x,r)^{2}\frac{dr}{r}}_{L^{p}(\sigma)}.
	\]
	This, \eqref{e:alphafmu<alphafsigma} and \eqref{eq:main theorem applied to mu} imply the desired estimate:
	\begin{equation*}
	\norm{f}_{L^{p}(\sigma)}\lec \nrm{\bigg(\int_{0}^{\infty}\theta_{\sigma}^f(x,r)^{2}\frac{dr}{r}\bigg)^{1/2}}_{L^{p}(\sigma)}.
	\end{equation*}
\end{proof}

\begin{proof}[Proof of the Main Theorem using \thmref{thm:dyadic main theorem}]
	By \lemref{lem:Sigma unbounded}, we may assume that $\supp\sigma=\Sigma$ is unbounded, so that \thmref{thm:dyadic main theorem} holds.

	Let $x\in\Sigma,\ r>0$. Let $k\in \Z$ be such that $\delta^{k+1} <r\le \delta^k$, and let $Q$ be a cube in $\D_k$ containing $x$. Recall that $Q\subset B(z_Q,3\ell(Q))$. Since $r\le\ell(Q)$, we have
	\begin{equation*}
	B(x,r)\subset B(z_Q, 3\ell(Q)+r)\subset B(z_Q,4\ell(Q))=B_Q.
	\end{equation*}
	Hence, by \eqref{eq:alphas on two balls},
	\begin{equation*}
	\alpha_{f\sigma}(x,r)\lesssim \alpha_{f\sigma}(Q).
	\end{equation*}
	We also have $|f|_{x,r}\lesssim |f|_{B_Q}$, and so
	\begin{equation*}
	|f|_{x,r}\alpha_{\sigma}(x,r)\lesssim |f|_{B_Q}\alpha_{\sigma}(Q).
	\end{equation*}
	Consequently, 
	\begin{equation*}
	\int_{\delta^{k+1}}^{\delta^k} (\alpha_{f\sigma}(x,r) + |f|_{x,r}\alpha_{\sigma}(x,r))^2\ \frac{dr}{r} \lesssim \alpha_{f\sigma}(Q)^2 + |f|_{B_Q}^2\alpha_{\sigma}(Q))^2.
	\end{equation*}
	Summing over $k\in\Z$ yields
	\begin{equation*}
	\int_0^\infty(\alpha_{f\sigma}(x,r) + |f|_{x,r}\alpha_{\sigma}(x,r))^2\ \frac{dr}{r} \lesssim \sum_{x\in Q\in\D}\alpha_{f\sigma}(Q)^2 + |f|_{B_Q}^2\alpha_{\sigma}(Q))^2.
	\end{equation*}
	
	Similarly, for $x\in\Sigma,\ r>0, \delta^{k+1}<r\le \delta^k$, we may consider a cube $Q\in\D_{k+2}$ such that $x\in Q\subset B_Q\subset B(x,r)$. Mimicking the estimates above, one gets
	\begin{equation*}
	\sum_{x\in Q\in\D}\alpha_{f\sigma}(Q)^2 + |f|_{B_Q}^2\alpha_{\sigma}(Q))^2 \lesssim \int_0^\infty(\alpha_{f\sigma}(x,r) + |f|_{x,r}\alpha_{\sigma}(x,r))^2\ \frac{dr}{r}.
	\end{equation*}
	Putting the two estimates together, we get the comparability of the dyadic and continuous variants of the square function:
	\begin{equation*}
	Jf(x)^2=\sum_{x\in Q\in\D}\alpha_{f\sigma}(Q)^2 + |f|_{B_Q}^2\alpha_{\sigma}(Q))^2 \sim \int_0^\infty(\alpha_{f\sigma}(x,r) + |f|_{x,r}\alpha_{\sigma}(x,r))^2\ \frac{dr}{r}.
	\end{equation*}	
\end{proof}

\thmref{thm:dyadic main theorem} will follow from the results from the next three sections. 
From now on we assume that $\sigma$ is a uniformly rectifiable measure with unbounded support.

\section{\texorpdfstring{$\nrm{Jf}_{2}\lec \nrm{f}_{2}$}{||Jf||\_2 < C||f||\_2}}\label{sec:Jf2 < f2}
	First, we prove the estimate $\norm{Jf}_p\lesssim \norm{f}_p$ in the case $p=2$.
	\begin{prop}\label{prop:alphas bounded by f}
		Let $f\in L^2(\sigma)$. Then
		\begin{equation*}
		\sum_{Q\in\D}(\alpha_{f\sigma}(Q)^2+|f|_{B_Q}^2\alpha_{\sigma}(Q)^2)\ell(Q)^d \lesssim \norm{f}_{L^2(\sigma)}^2.
		\end{equation*}
	\end{prop}
	
	Our main tool in the proof of \propref{prop:alphas bounded by f} are the martingale difference operators associated to systems of cubes $\D(\omega$).
	
	Given $\omega\in\Omega$, $Q\in\D(\omega),$ and $f\in L^1_{loc}(\sigma)$ we set
	\begin{equation*}
	\Delta_Q f = \sum_{P\in \Ch(Q)}\avg{f}_P\one_P - \avg{f}_Q\one_Q.
	\end{equation*}
	Observe that all $\Delta_Q f$ have zero mean, i.e. $\int \Delta_Q f\ d\sigma = 0.$
	
	It is well known (see e.g. \cite[Chapter 6.4]{grafakos2014classical}) that given $f\in L^2(\sigma)$ and some system of cubes $\D(\omega)$ we have
	\begin{equation*}
	f=\sum_{Q\in\D(\omega)}\Delta_Q f
	\end{equation*}
	with the convergence understood in the $L^2$ sense. It is crucial that $\sigma(\Sigma)=\infty$, so that $f+C\in L^2(\sigma)$ if and only if $C=0$ (in the case $\sigma(\Sigma)<\infty$ one would have to subtract from the left hand side above the average of $f$).
	
	Note that $\Delta_Q f$ are mutually orthogonal in $L^2(\sigma)$, so that
	\begin{equation}\label{eq:L2 of martingale bounded by L2 of fcn}
	\norm{f}_{L^2(\sigma)}^2 = \sum_{Q\in\D(\omega)}\norm{\Delta_Q f}_{L^2(\sigma)}^2
	\end{equation}

	Moreover, if $Q\in\D(\omega)$, then for $\sigma$-a.e. $x\in Q$
	\begin{equation}\label{eq:martingale identity}
	f(x) = \avg{f}_Q + \sum_{P\in\D(Q)}\Delta_P f(x).
	\end{equation}
	
	\begin{lemma}\label{lem:key estimate}
		Suppose $Q\in\D,$ and let $R=R(Q)\in\D(\omega(Q))$ be as in Section 2.1. Then,  for $f\in L^2(\sigma)$ we have
		\begin{equation}\label{eq:alpha fsigma estimate}
		\alpha_{f\sigma}(Q)\lesssim |\avg{f}_{R}|\alpha_{\sigma}(R) + \sum_{P\in\D(R)}\frac{\ell(P)^{1+d/2}}{\ell(Q)^{1+d}}\norm{\Delta_P f}_2.
		\end{equation}
	\end{lemma}
	\begin{proof}
		Let $\varphi\in \lip_1(B_Q)$ and consider a candidate for $\cL^{f\sigma}_Q$ of the form $\avg{f}_R\cL^{\sigma}_Q$. For all $x\in B_Q\cap\supp\sigma$ we have $x\in R$, so that using \eqref{eq:martingale identity}
		\begin{multline*}
		\abs{\int\varphi(x)f(x)\ d\sigma(x) - \avg{f}_R\int\varphi(x)\ d\cL^{\sigma}_Q(x)} \\
		= \abs[2]{\int\varphi(x)\avg{f}_R + \sum_{P\in\D(R)}\varphi(x)\Delta_P f(x)\ d\sigma(x) - \int\varphi(x) \avg{f}_R\ d\cL^{\sigma}_Q(x)}\\
		\le \abs{\avg{f}_R}\abs{\int\varphi(x)\ d\sigma(x) - \int\varphi(x)\ d\cL^{\sigma}_Q(x)} + \sum_{P\in\D(R)} \abs[2]{\int \varphi(x)\Delta_P f(x)\ d\sigma(x)}
		 =: I_1 + I_2.
		\end{multline*}
		It is clear that 
		\[
		I_1\le \abs{\avg{f}_R}\alpha_{\sigma}(Q)\ell(Q)^{d+1}\overset{\eqref{eq:alphas on two balls}}{\lec} \abs{\avg{f}_R}\alpha_{\sigma}(R)\ell(Q)^{d+1},\] 
		which gives rise to the first term on the right hand side of \eqref{eq:alpha fsigma estimate}.
		
		Concerning $I_2$, we use the zero mean property of martingale difference operators, and the fact that $\varphi\in\lip_1(B_Q)$, to get
		\begin{multline*}
		I_2 = \sum_{P\in\D(R)} \abs[2]{\int \big(\varphi(x) - \varphi(z_P)\big)\Delta_P f(x)\ d\sigma(x)}\le \sum_{P\in\D(R)}\int \abs{\varphi(x) - \varphi(z_P)}\abs{\Delta_P f(x)}\ d\sigma(x)\\
		\lesssim \sum_{P\in\D(R)} \ell(P)\norm{\Delta_P f}_1\overset{\text{H\"{o}lder}}{\lesssim} \sum_{P\in\D(R)} \ell(P)^{1+d/2}\norm{\Delta_P f}_2.
		\end{multline*}
		Dividing by $\ell(Q)^{d+1}$ and taking supremum over $\varphi\in\lip_1(B_Q)$ yields \eqref{eq:alpha fsigma estimate}.
	\end{proof}
	
	\begin{proof}[Proof of \propref{prop:alphas bounded by f}]
		We begin by noting that, since $\sigma$ is uniformly rectifiable, $\alpha_{\sigma}(Q)^2\ell(Q)^d$ is a Carleson measure by the results from \cite{Tol09}, see \thmref{t:xavi}. Therefore, the estimate
		\begin{equation*}
		\sum_{Q\in\D} |f|_{B_Q}^2\alpha_{\sigma}(Q)^2\ell(Q)^d \lesssim \norm{f}_{L^2(\sigma)}^2
		\end{equation*}
		follows immediately from Carleson's embedding theorem, see e.g. \cite[Theorem 5.8]{tolsa2014analytic}, and we only need to estimate the sum involving $\alpha_{f\sigma}(Q)$.
		
		Observe that for each $\omega\in\Omega$ and $R\in \D(\omega)$ there is at most a bounded number of cubes $Q\in\D$ such that $R(Q)=R$. 
		
		Fix some $\omega\in\Omega$. Recall that $\cG(\omega)$ is the family of cubes $Q\in\D$ such that $\omega(Q)=\omega$. We apply \eqref{eq:alpha fsigma estimate} and the observation above to get
		\begin{multline*}
		\sum_{Q\in\cG(\omega)} \alpha_{f\sigma}(Q)^2\ell(Q)^d\lesssim \sum_{R\in\D(\omega)}|\avg{f}_{R}|^2\alpha_{\sigma}(R)^2\ell(R)^d + \sum_{R\in\D(\omega)}\bigg(\sum_{P\in\D(R)}\frac{\ell(P)^{1+d/2}}{\ell(R)^{1+d/2}}\norm{\Delta_P f}_2\bigg)^2\\ =: S_1 + S_2.
		\end{multline*}
		
		Concerning $S_1$, we may use Carleson's embedding theorem again to estimate $S_1\lesssim \norm{f}_2^2$.
		
		Moving on to $S_2$, we apply Cauchy-Schwarz inequality to get
		\begin{equation*}
		S_2\le \sum_{R\in\D(\omega)}\bigg(\sum_{P\in\D(R)}\frac{\ell(P)}{\ell(R)}\norm{\Delta_P f}^2_2\bigg)\bigg(\sum_{P\in\D(R)}\frac{\ell(P)^{d+1}}{\ell(R)^{d+1}}\bigg).
		\end{equation*}
		It is easy to see that, due to Ahlfors regularity of $\sigma$, $\sum_{P\in\D(R)}\frac{\ell(P)^{d+1}}{\ell(R)^{d+1}}\lesssim 1$. Thus,
		\begin{equation*}
		S_2\le \sum_{R\in\D(\omega)}\sum_{P\in\D(R)}\frac{\ell(P)}{\ell(R)}\norm{\Delta_P f}^2_2 = \sum_{P\in\D(\omega)}\norm{\Delta_P f}^2_2 \sum_{\substack{R\in\D(\omega)\\ R\supset P}}\frac{\ell(P)}{\ell(R)}\lesssim \sum_{P\in\D(\omega)}\norm{\Delta_P f}^2_2 \overset{\eqref{eq:L2 of martingale bounded by L2 of fcn}}{\lesssim} \norm{f}_2^2.
		\end{equation*}
		
		Putting the estimates above together we arrive at
		\begin{equation*}
		\sum_{Q\in\cG(\omega)}\alpha_{f\sigma}(Q)^2\ell(Q)^d\lesssim \norm{f}_2^2.
		\end{equation*}
		Summing over all $\omega\in\Omega$ (recall that $\#\Omega$ is bounded) we get the desired estimate.
	\end{proof}
	
	\section{\texorpdfstring{$\nrm{Jf}_{p}\lec \nrm{f}_{p}$ for $1<p<\infty$}{||Jf||\_p < C||f||\_p for 1<p<infty}}\label{sec:Jfp < fp}
	In this section we use the estimate $||Jf||_{2}\lec \nrm{f}_{2}$ to prove $||Jf||_{p}\lec \nrm{f}_{p}$ for general $1<p<\infty$. More precisely, we will show a localized version of the estimate, which implies the global estimate via a limiting argument.
	
	Fix an arbitrary $Q_0\in\D$ and set
	\begin{equation*}
	J_{0}f(x) := \bigg( \sum_{x\in Q\in\D(Q_0)} \alpha_{f\sigma}(Q)^2 + |f|_{B_Q}^2\alpha_{\sigma}(Q)^2\bigg)^{1/2}.
	\end{equation*}
	
	\begin{prop}\label{prop:local Lp estimate}
		Let $1<p<\infty$ and $f\in L^p(\sigma)$. Then,
		\begin{equation*}
		\norm{J_0 f}_{L^p(Q_0)}\lesssim_p \norm{f}_{L^p(B_{Q_0})}.
		\end{equation*}
	\end{prop}
	The proposition follows easily from a good-lambda inequality stated below. Let $M$ denote the non-centered maximal Hardy-Littlewood operator with respect to $\sigma$, i.e.
	\begin{equation*}
	Mf(x) = \sup\{|f|_B\ :\ x\in B,\ B\ \text{is a ball}\}.
	\end{equation*}
Since $\sigma$ is Ahlfors regular, the operator $M$ is bounded on $L^p(\sigma)$ for $p>1$, see e.g. \cite[Theorem 2.6, Remark 2.7]{tolsa2014analytic}.
	\begin{lemma}\label{lem:good lambda}
Let $f\in L^1_{loc}(\sigma)$. For any $\alpha>1$ there exists $\varepsilon=\varepsilon(\alpha)>0$ such that for all $\lambda>0$
		\begin{equation}\label{eq:good lambda}
		\sigma(\{x\in Q_0\ :\ J_0f(x)>\alpha\lambda,\ Mf(x)\le \varepsilon\lambda \})
		 \le \frac{9}{10} \sigma(\{x\in Q_0\ :\ J_0f(x)>\lambda \}).
		\end{equation}
	\end{lemma}
	Let us show how to use the above to prove \propref{prop:local Lp estimate}.
	\begin{proof}[Proof of \propref{prop:local Lp estimate}]
		Note that $J_0f = J_0(f\one_{B_{Q_0}})$, so without loss of generality we may assume that $\supp f\subset B_{Q_0}$. Let $\alpha=\alpha(p)>1$ be so close to $1$ that $0.9\alpha^p<0.95$, and let $\varepsilon=\varepsilon(\alpha)$ be as in \lemref{lem:good lambda}. We use the layer cake representation to get
		\begin{multline*}
		\int_{Q_0} J_0f(x)^p\ d\sigma(x) = p \int_0^{\infty} \lambda^{p-1}\sigma(\{x\in Q_0\ :\ J_0f(x)>\lambda \})\ d\lambda\\
		 = p\alpha^p \int_0^{\infty} \lambda^{p-1}\sigma(\{x\in Q_0\ :\ J_0f(x)>\alpha\lambda \})\ d\lambda\\
		 \le p\alpha^p \int_0^{\infty} \lambda^{p-1}\sigma(\{x\in Q_0\ :\ J_0f(x)>\alpha\lambda,\ Mf(x)\le\varepsilon\lambda \})\ d\lambda\\
		  \qquad\qquad\qquad + p\alpha^p \int_0^{\infty} \lambda^{p-1}\sigma(\{x\in Q_0\ :\ Mf(x)>\varepsilon\lambda \})\ d\lambda\\
		  \overset{\eqref{eq:good lambda}}{\le} \frac{9}{10} p\alpha^p \int_0^{\infty} \lambda^{p-1}\sigma(\{x\in Q_0\ :\ J_0f(x)>\lambda \})\ d\lambda + \alpha^p\varepsilon^{-p} \int_{Q_0} Mf(x)^p\ d\sigma(x)\\
		  \le \frac{19}{20} p \int_0^{\infty} \lambda^{p-1}\sigma(\{x\in Q_0\ :\ J_0f(x)>\lambda \})\ d\lambda + \alpha^p\varepsilon^{-p} \int_{Q_0} Mf(x)^p\ d\sigma(x)\\
		  = \frac{19}{20}\int_{Q_0} J_0f(x)^p\ d\sigma(x) + \alpha^p\varepsilon^{-p} \int_{Q_0} Mf(x)^p\ d\sigma(x).
		\end{multline*}
		Absorbing the first term from the right hand side into the left hand side, we arrive at
		\begin{equation*}
		\int_{Q_0} J_0f(x)^p\ d\sigma(x) \le 20 \alpha^p\varepsilon^{-p} \int Mf(x)^p\ d\sigma(x).
		\end{equation*}
		We use the $L^p$ boundedness of $M$ and the assumption $\supp f\subset B_{Q_0}$ to conclude
		\begin{equation*}
		\int_{Q_0} J_0f(x)^p\ d\sigma(x) \lesssim_{\alpha,\varepsilon} \int_{B_{Q_0}} f(x)^p\ d\sigma(x).
		\end{equation*}
	\end{proof}
	The remainder of this section is dedicated to proving \lemref{lem:good lambda}. 
	\subsection{Preliminaries}
	Fix $\alpha>1$ and $\lambda>0$. First, we set
		\begin{equation*}
		E_{\lambda} = \{x\in Q_0\ :\ J_0f(x)>\lambda \}.
		\end{equation*}
	Consider the covering of $E_{\lambda}$ with a family of cubes $\sC_{\lambda}\subset \D(Q_0)$ such that for every $S \in \sC_{\lambda}$ we have
	\begin{equation*}
	\sigma(S\cap E_{\lambda})\ge 0.99\sigma(S)
	\end{equation*}
	and $S$ is the maximal cube with this property. Since the cubes from $\sC_{\lambda}$ are pairwise disjoint, to get \eqref{eq:good lambda} it is enough to find $\varepsilon=\varepsilon(\alpha)$ such that for each $S\in\sC_{\lambda}$ we have
		\begin{equation}\label{eq:good lambda in R}
		\sigma(\{x\in S\ :\ J_0f(x)>\alpha\lambda,\ Mf(x)\le \varepsilon\lambda \})
		\le \frac{8}{10} \sigma(S).
		\end{equation}
	Fix $S\in\sC_{\lambda}$. Without loss of generality assume that
	\begin{equation}\label{eq:large set with maximal bdd}
	\sigma(\{x\in S\ :\ Mf(x)\le\varepsilon\lambda \})
	> \frac{8}{10} \sigma(S),
	\end{equation}
	otherwise there is nothing to prove.
	
	Given $x\in S$, we split the sum from the definition of $J_0f(x)$ into two parts:
	\begin{multline}\label{eq:J0 decomposition}
	J_{0}f(x)^2 = \sum_{x\in Q\in\D(S)} \big(\alpha_{f\sigma}(Q)^2 + |f|_{B_Q}^2\alpha_{\sigma}(Q)^2\big)
	 + \sum_{S\subsetneq Q\in\D(Q_0)}\big( \alpha_{f\sigma}(Q)^2 + |f|_{B_Q}^2\alpha_{\sigma}(Q)^2\big)\\ =: J_1 f(x)^2 + J_2 f(x)^2.
	\end{multline}
	Clearly, $J_2f(x)\equiv J_2 f$ is just a constant. By the definition of $\sC_{\lambda}$ there exists  $y\in \hat{S}$ (where $\hat{S}$ is the parent of $S$) such that $y\not\in E_{\lambda}$.  By the definition of $E_{\lambda}$, we get that
	\begin{equation*}
	J_2f\le J_0 f(y)\le\lambda.
	\end{equation*}
	
	We will show the following.
	\begin{lemma}\label{lem:set S'}
		There exists a set $S_1\subset S$ such that $\sigma(S_1)\ge 0.5 \sigma(S)$ and 
		\begin{equation*}
		\int_{S_1} J_1 f(x)^2\ d\sigma(x)\lesssim \varepsilon^2\lambda^2\sigma(S_1)
		\end{equation*}
	\end{lemma}
	The estimate \eqref{eq:good lambda in R} follows from the above easily. Indeed, using Chebyshev, we can find $S_2\subset S_1$ such that for all $x\in S_2$ we have $J_1 f(x)\lesssim\varepsilon\lambda$ and $\sigma(S_2)\ge0.5\sigma(S_1)\ge 0.2\sigma(S)$. Then, choosing $\varepsilon=\varepsilon(\alpha)$ small enough, \eqref{eq:J0 decomposition} gives $J_0f(x)^2\le\lambda^2 + C\varepsilon^2\lambda^2\le \alpha^2\lambda^2$ on $S_2$, so that
	\begin{equation*}
	\sigma(\{x\in S\ :\ J_0f(x)>\alpha\lambda,\ Mf(x)\le\varepsilon\lambda \})
	\le \sigma(S\setminus S_2)\le  \frac{8}{10}\sigma(S).
	\end{equation*}
	So our goal is to prove \lemref{lem:set S'}.
	
	\subsection{\texorpdfstring{Calder\'{o}n}{Calderon}-Zygmund decomposition}
	Let $R=R(S)$ be as in Section \ref{sec:adjacent}, so that $B_S\cap \supp\sigma\subset R$. We consider a variant of the Calder\'{o}n-Zygmund decomposition of $f\one_R$ with respect to $\D(R)$ at the level $2\varepsilon\lambda$.
	
	First, let $\{Q_j\}_j\subset\D(R)$ be maximal cubes satisfying $|f|_{B_{Q_j}}\ge 2 \varepsilon\lambda$.
	Note that for all $x\in Q_j$ (and recalling that $M$ is the non-centered maximal function) we have
	\begin{equation*}
	Mf(x)\ge |f|_{B_{Q_j}}\ge 2\varepsilon\lambda.
	\end{equation*}
	Hence, $\bigcup_j Q_j\subset \{x\in S\ :\ Mf(x)\ge 2\varepsilon\lambda\}$, and so 
	\begin{multline}\label{eq:size of Q_j 1}
	\sigma(R\setminus\bigcup_j Q_j)\ge \sigma(S\setminus\bigcup_j Q_j)\ge \sigma(\{x\in S\ :\ Mf(x)\le \varepsilon\lambda \})\\
	\overset{\eqref{eq:large set with maximal bdd}}{\ge}  \frac{8}{10}\sigma(S)\sim \ell(S)^d\sim \ell(R)^d.
	\end{multline}
	In particular, $Q_j\neq R$ for all $j$. Thus, by the maximality of $Q_j$ we get easily
	\begin{equation}\label{eq:f avg on Qj}
	{|f|}_{B_{Q_j}}\sim \varepsilon\lambda.
	\end{equation}
	We define $g\in L^{\infty}(\sigma)$ by 
	\begin{equation*}
	g(x) = f(x)\one_{R\setminus\bigcup_j Q_j}(x)+\sum_j\avg{f}_{Q_j}\one_{Q_j}(x).
	\end{equation*} 
	From the definition of $Q_j$ and \eqref{eq:f avg on Qj} it follows that  $\norm{g}_{\infty}\lesssim \varepsilon\lambda$.
	We define also $b\in L^1(\sigma)$ as
	\begin{equation*}
	b(x)= \sum_j (f(x)-\avg{f}_{Q_j})\one_{Q_j}(x)=:\sum_j b_j(x).
	\end{equation*}
	Note that $f=g+b$ and for all $j$ we have $\int b_j\ d\sigma = 0$.
	
	\subsection{Definition of \texorpdfstring{$S_1$}{S\_1}}
	We set $S_1=S\setminus N_{\eta}$, where $N_{\eta}$ is some small neighbourhood of $\bigcup_j Q_j.$ To make this more precise, given a small $\eta>0$ we define $N_{\eta} = \bigcup_j N_{\eta,j}$, where
	\begin{equation*}
	N_{\eta,j} = \{x\in\supp\sigma\ :\ \dist(x,Q_j)<\eta\ell(Q_j) \}.
	\end{equation*}
	The thin boundaries property of $\D$ \eqref{eq:thin boundaries} gives
	\begin{equation*}
	\sigma(	N_{\eta,j}\setminus Q_j)\le \eta^{\gamma}\sigma(Q_j)
	\end{equation*}
	for some $\gamma\in (0,1)$. From \eqref{eq:size of Q_j 1} and the fact that $\sigma(S)\sim\sigma(R)$ we get
	\begin{multline*}
	\sigma(S\setminus N_{\eta}) \ge \sigma(S\setminus\bigcup_j Q_j) - \sum_j \sigma(N_{\eta,j}\setminus Q_j)\overset{\eqref{eq:size of Q_j 1}}{\ge} \frac{8}{10}\sigma(S) - \sum_j \eta^{\gamma}\sigma(Q_j)\\
	\ge \frac{8}{10}\sigma(S) - \eta^{\gamma} \sigma(R)\ge \frac{8}{10}\sigma(S) - C \eta^{\gamma} \sigma(S) = \left(\frac{8}{10} - C \eta^{\gamma}\right)\sigma(S).
	\end{multline*}
	Here $C$ depends only on the implicit constant in $\sigma(S)\sim\sigma(R)$, which in turn depends on the Ahlfors regularity constant of $\sigma$ and on the parameters from the definition of the system $\D$. 
	
	Choosing $\eta$ so small that $C\eta^{\gamma}<0.1$, we get that $S_1=S\setminus N_{\eta}$ satisfies 
	\begin{equation*}
	\sigma(S_1)\ge \frac{7}{10}\sigma(S).
	\end{equation*}

	\subsection{Estimating \texorpdfstring{$J_1f$}{J\_1 f}}
	Now, we will show that
	\begin{equation}\label{eq:J1f estimate on S1}
	\int_{S_1} J_1 f(x)^2\ d\sigma(x)\lesssim \varepsilon^2\lambda^2\sigma(S_1)
	\end{equation}
	Recall that 
	\begin{equation*}
	J_1 f(x)^2 = \sum_{x\in Q\in\D(S)} \alpha_{f\sigma}(Q)^2 + \sum_{x\in Q\in\D(S)} |f|_{B_Q}^2\alpha_{\sigma}(Q)^2 =: J_{1}'f(x)^2 + J_{1}''f(x)^2.
	\end{equation*}
	
	First we deal with $J_{1}''f$. Observe that for all $Q\in\D(S)$ intersecting $S_1$ we have
	\begin{equation}\label{eq:average of f over BQ}
	{|f|}_{B_Q}\lesssim \varepsilon\lambda.
	\end{equation}
	Indeed, let $y\in Q\cap S_1$, and let $P\in\D(R)$ be such that $y\in P$, $\ell(Q)\sim\ell(P)$, and $B_Q\subset B_{P}$. By the maximality of $Q_j$ and the fact that $P\setminus \bigcup_j Q_j\neq\varnothing$ we get $|f|_{B_{P}}\le 2\varepsilon\lambda$. Estimate \eqref{eq:average of f over BQ} follows from the inclusion $B_Q\subset B_P$.
	
	Using \eqref{eq:average of f over BQ} as well as \thmref{t:xavi} we get
	\begin{multline*}
	\int_{S_1} \sum_{x\in Q\in\D(S)}{|f|}_{B_Q}^2\alpha_{\sigma}(Q)^2\ d\sigma(x) 
	\lesssim \varepsilon^2\lambda^2\sum_{Q\in\D(S)}\alpha_{\sigma}(Q)^2\sigma(Q\cap S_1)
	\\\lesssim \varepsilon^2\lambda^2\sum_{Q\in\D(S)}\alpha_{\sigma}(Q)^2\sigma(Q)\lesssim \varepsilon^2\lambda^2\sigma(S)\sim \varepsilon^2\lambda^2\sigma(S_1).
	\end{multline*}
	Thus, we are only left with showing
	\begin{equation}
	\label{eq:J_1'}
	\int_{S_1} J_{1}'{f}(x)^2\ d\sigma(x)=\int_{S_1} \sum_{x\in Q\in\D(S)} \alpha_{f\sigma}(Q)^2\ d\sigma(x)\lesssim \varepsilon^2\lambda^2\sigma(S_1).
	\end{equation}
	
	\begin{lemma}\label{lem:estimate of alpha fsigma}
		For $Q\in\D(S)$ we have	
		\begin{equation*}
		\alpha_{f\sigma}(Q)\lesssim \alpha_{g\sigma}(Q) + \varepsilon\lambda\sum_{j: Q_j\cap B_Q\neq\varnothing} \frac{\ell(Q_j)^{d+1}}{\ell(Q)^{d+1}}.
		\end{equation*}
	\end{lemma}
	\begin{proof}
		Let $\varphi\in\lip_1(B_Q)$. Then, using the decomposition $f(y)=g(y)+b(y)$ valid for all $y\in R\supset B_S\cap\supp\sigma\supset B_Q\cap\supp\sigma$, 
	\begin{multline*}
	\abs{\int\varphi(y)f(y)\ d\sigma(y) - \int\varphi(y)\ d\cL_Q^{g\sigma}(y)} \\
	\le \abs{\int\varphi(y)g(y)\ d\sigma(y) - \int\varphi(y)\ d\cL_Q^{g\sigma}(y)} + \abs{\int \varphi(y) b(y)\ d\sigma(y)}\\
	\lesssim \ell(Q)^{d+1}\alpha_{g\sigma}(Q) + \sum_j\abs{\int \varphi(y) b_j(y)\ d\sigma(y)}.
	\end{multline*}
	Concerning the second term on the right hand side, recall that $\int b_j\ d\sigma= 0$ and that $\supp b_j\subset Q_j$. Keeping that in mind, denoting by $x_j$ the center of $Q_j$, we estimate in the following way:
	\begin{multline*}
	\sum_j\abs{\int \varphi(y) b_j(y)\ d\sigma(y)}= \sum_j\abs{\int (\varphi(y)-\varphi(x_j)) b_j(y)\ d\sigma(y)}\\ \le \sum_j\int \abs{(\varphi(y)-\varphi(x_j)) b_j(y)}\ d\sigma(y)
	\lesssim \sum_{j: Q_j\cap B_Q\neq\varnothing} \ell(Q_j)\int \abs{b_j(y)}\ d\sigma(y)\\
	= \sum_{j: Q_j\cap B_Q\neq\varnothing} \ell(Q_j)\int_{Q_j} \abs{f(y)-\avg{f}_{Q_j}}\ d\sigma(y)\lesssim \sum_{j: Q_j\cap B_Q\neq\varnothing} \ell(Q_j)^{d+1}\avg{|f|}_{Q_j}\\
	\overset{\eqref{eq:f avg on Qj}}{\lec} \varepsilon\lambda\sum_{j: Q_j\cap B_Q\neq\varnothing} \ell(Q_j)^{d+1}.
	\end{multline*}
	Together with the previous string of estimates, taking supremum over all $\varphi\in\lip_1(B_Q)$, we get 
	\begin{equation*}
	\alpha_{f\sigma}(Q)\lesssim \alpha_{g\sigma}(Q) + \varepsilon\lambda\sum_{j: Q_j\cap B_Q\neq\varnothing} \frac{\ell(Q_j)^{d+1}}{\ell(Q)^{d+1}}.
	\end{equation*}	
	\end{proof}
	An immediate consequence of \lemref{lem:estimate of alpha fsigma} is the estimate
	\begin{multline}\label{eq:J_1f on S'}
	\int_{S_1} J_1' f(x)^2\ d\sigma(x)\\
	\lesssim \int_{S_1} J_1' g(x)^2\ d\sigma(x) + \varepsilon^2\lambda^2 \int_{S_1}\sum_{x\in Q\in\D(S)}\left( \sum_{j: Q_j\cap B_Q\neq\varnothing} \frac{\ell(Q_j)^{d+1}}{\ell(Q)^{d+1}}\right)^2\ d\sigma(x).
	\end{multline}

	Using \propref{prop:alphas bounded by f} and the fact that $\norm{g}_{\infty}\lesssim\varepsilon\lambda,\ \supp g\subset R,$ we get
	\begin{equation}\label{eq:J_1 g estimated}
	\int_{S_1} {J_1'g(x)^2}\ d\sigma(x)\le\norm{Jg}_2^2\lesssim \norm{g}_2^2\le\norm{g}_{\infty}^{2}\sigma(R)\lesssim \varepsilon^2\lambda^2\sigma(R)\sim \varepsilon^2\lambda^2\sigma(S_1).
	\end{equation}
	
	Moving on to the second term from the right hand side of \eqref{eq:J_1f on S'}, denote by $\Tree\subset\D(S)$ the family of cubes contained in $S$ that intersect $S_1$. We have
	\begin{multline}\label{eq:the ugly sum}
	\int_{S_1} \sum_{x\in Q\in\D(S)}\left( \sum_{j: Q_j\cap B_Q\neq\varnothing} \frac{\ell(Q_j)^{d+1}}{\ell(Q)^{d+1}}\right)^2\ d\sigma(x) \le \sum_{Q\in\Tree}\sigma(Q)\left( \sum_{j: Q_j\cap B_Q\neq\varnothing} \frac{\ell(Q_j)^{d+1}}{\ell(Q)^{d+1}}\right)^2\\
	\overset{\text{Cauchy-Schwarz}}{\lesssim} \sum_{Q\in\Tree}\ell(Q)^{-d-2}\left( \sum_{j: Q_j\cap B_Q\neq\varnothing} {\ell(Q_j)^{d+2}}\right)\left( \sum_{j: Q_j\cap B_Q\neq\varnothing} {\ell(Q_j)^{d}}\right)
	\end{multline}
	Note that since $Q\in\Tree$, we have $Q\cap S_1\neq\varnothing$. By the definition of $S_1$, this implies that for all $j$ such that $Q_j\cap B_Q\neq\varnothing$ we have $\ell(Q)\gtrsim_{\eta}\ell(Q_j)$. Indeed, if $\ell(Q)\ll \eta\ell(Q_j)$, then $B_Q\cap Q_j\neq\varnothing$ implies $Q\subset N_{\eta,j}$, which would contradict $Q\cap S_1\neq\varnothing$.
	
	By the observation above, we have some $C=C(\eta)$ such that if $B_Q\cap Q_j\neq\varnothing$, then $Q_j\subset CB_Q$. Consequently, 
	\begin{equation*}
	\sum_{j: Q_j\cap B_Q\neq\varnothing} {\ell(Q_j)^{d}}\lesssim \sum_{j: Q_j\subset CB_Q} \sigma(Q_j) \le \sigma(CB_Q)\sim_{\eta} \ell(Q)^d.
	\end{equation*}
	Thus, the right hand side of \eqref{eq:the ugly sum} can be estimated by
	\begin{equation}\label{eq:sum over tree}
	\sum_{Q\in\Tree}\ell(Q)^{-2} \sum_{j: Q_j\cap B_Q\neq\varnothing} {\ell(Q_j)^{d+2}} = \sum_{j} {\ell(Q_j)^{d+2}} \sum_{Q\in\Tree : Q_j\cap B_Q\neq\varnothing}\ell(Q)^{-2}.
	\end{equation}
	As noted above, $ Q_j\cap B_Q\neq\varnothing$ implies $\ell(Q)\gtrsim_{\eta}\ell(Q_j)$. Hence, 
	\begin{equation*}
	\sum_{Q\in\Tree : Q_j\cap B_Q\neq\varnothing}\ell(Q)^{-2}\lesssim_{\eta}\ell(Q_j)^{-2},
	\end{equation*}
	where we used the fact that the sum above is essentially a geometric series. Putting this together with \eqref{eq:sum over tree} and \eqref{eq:the ugly sum}, we get
	\begin{equation*}
	\int_{S_1} \sum_{x\in Q\in\D(S)}\left( \sum_{j: Q_j\cap B_Q\neq\varnothing} \frac{\ell(Q_j)^{d+1}}{\ell(Q)^{d+1}}\right)^2\ d\sigma(x)\lesssim_{\eta}\sum_j\ell(Q_j)^d\lesssim\ell(R)^d\sim\sigma(S_1).
	\end{equation*}
	

	Together with \eqref{eq:J_1f on S'} and \eqref{eq:J_1 g estimated} this gives the desired estimate \eqref{eq:J_1'}:
	\begin{equation*}
	\int_{S_1} J_1' f(x)^2\ d\sigma(x)\lesssim_{\eta} \varepsilon^2\lambda^2\sigma(S_1).
	\end{equation*}
	
	This finishes the proof of \lemref{lem:set S'}.
	
\section{\texorpdfstring{$\nrm{f}_{p}\lec \nrm{Jf}_{p}$ for $1<p<\infty$}{||f||\_p < C||Jf||\_p for 1<p<infty}}\label{sec:fp < Jfp}

In this section we show the second inequality of \thmref{thm:dyadic main theorem}.

\begin{prop}
	Let $f\in L^p(\sigma)$ for some $1<p<\infty$. Then
	\begin{equation}\label{eq:f smaller than Jf}
	\norm{f}_{L^p(\sigma)} \lesssim \norm{Jf}_{L^p(\sigma)}.
	\end{equation}
\end{prop}

\subsection{Littlewood-Paley theory}
Our main tool will be the Littlewood-Paley theory for spaces of homogeneous type developed by David, Journ\'{e} and Semmes in \cite{DJS85}. We follow the way it was paraphrased (in English) in \cite[Section 15]{tolsa2017rectifiable}.

For $r>0$, $x\in\Sigma$, and $g\in L^1_{loc}(\sigma)$, let 
\[D_r g(x)=\frac{\phi_r*(g\sigma)(x)}{\phi_r*\sigma(x)}
- \frac{\phi_{2r}*(g\sigma)(x)}{\phi_{2r}*\sigma(x)}\]
where $\phi_r(y)=r^{-d}\phi(y/r)$ and $\phi$ is a radially symmetric smooth nonnegative function supported in $B(0,1)$ with $\int_{\R^{n}}\phi=1$.

For a function $g\in L^1_{loc}(\sigma)$ and $r>0$, we denote
$$S_r g(x) = \frac{\phi_r*(g\sigma)(x)}{\phi_r*\sigma(x)},$$
so that 
\[
D_rg = S_rg -S_{2r}g.
\]

Let $W_r$ be the operator of multiplication by $1/S_r^*1$. We consider the operators
\[\tilde{S}_{r} = S_r\,W_r\,S_r^* 
\;\; \mbox{ and } \;\; 
\tilde{D}_{r} = \tilde{S}_{r}-\tilde{S}_{2r}.\] 
Note that $\tilde{S}_{r}$, and thus $\tilde{D}_{r}$, are self-adjoint and
$\tilde{S}_{r}1 \equiv1$, so that 
\begin{equation}
\label{e:D1=0}
\tilde{D}_{r}1 = \tilde{D}_{r}^*1 =0.
\end{equation}

Let $s_r(x,y)$ the kernel of $S_r$ with respect to $\sigma$, that is, so we can write 
\[
S_r g(x)  =\int s_r(x,y)\,g(y)\,d\sigma(y).\] 
Observe that
\[s_r(x,y) = \frac1{\phi_r*\sigma(x)}\,\phi_r(x-y)\]
and the kernel of $\tilde{S}_r$ is
\[\tilde{s}_{r}(x,y) = \int s_r(x,z)\,\frac1{S_r^*1(z)}\,s_r(y,z)\,d\sigma(z).\]
{
We claim that the kernel $\tilde{d}_{r}(x,\cdot)$ for the operator $\tilde{D}_{r}$ is supported in $B(x,4r)$ and satisfies the Lipschitz bounds
\begin{equation}
\label{e:dklip}
|\tilde{d}_{r}(x,y)-\tilde{d}_{r}(x,z)|\lec |y-z|r^{-d-1}. 
\end{equation}

Indeed, let $x,x'\in \supp \sigma$. Since $\phi_r$ is $Cr^{-d-1}$-Lipschitz and $\sigma$ is Ahlfors regular,
\begin{align*}
|\phi_{r}*\sigma(x)-\phi_{r}*\sigma(x')|
& =\av{\int (\phi_{r}(x-y)-\phi_{r}(x'-y))d\sigma(y)}
\lesssim \frac{|x-x'|}{r^{d+1}}\sigma(B(x,r)\cup B(x',r))\\
& \lesssim \frac{|x-x'|}{r}.
\end{align*}
Thus, for $y\in \supp \sigma$,
\begin{align*}
|s_{r}(x,y)-s_{r}(x',y)|
& \leq \frac{|\phi_{r}(x-y)-\phi_{r}(x'-y)|}{\phi_{r}*\sigma(x)}
+\phi_{r}(x'-y) \left| \frac{1}{\phi_{r}*\sigma(x)}-\frac{1}{\phi_{r}*\sigma(x')}\right|\\
& \lesssim \frac{|x-x'|}{r^{d+1}}+r^{-d}\frac{|\phi_{r}*\sigma(x)-\phi_{r}*\sigma(x')|}{\phi_{r}*\sigma(x)^2}
\sim \frac{|x-x'|}{r^{d+1}}.
\end{align*}
Hence, 
\begin{align*}
|\tilde{s}_{r}(x,y)-\tilde{s}_{r}(x',y)
& = \av{\int (s_r(x,z)-s_{r}(x',z)) \frac1{S_r^*1(z)}\,s_r(y,z)\,d\sigma(z)}\\
& \lesssim \frac{|x-x'|}{r^{d+1}}  \av{\int  \frac1{S_r^*1(z)}\,s_r(y,z)\,d\sigma(z)}
 \lesssim \frac{|x-x'|}{r^{d+1}} 
\end{align*}
where in the last line we used the fact that $\int s_{r}(y,z)d\sigma(z)=1$ and 
\[
S_{r}^{*}1(z) = \int \frac{\phi_{r}(x-z)}{\phi_{r}*\sigma(x)}d\sigma(x)
\geq \int_{B(z,r/2)} \frac{r^{-d}}{\phi_{r}*\sigma(x)}d\sigma(x)
\sim 1.
\]

Since $\tilde{d}_{r} = \tilde{s}_{r}-\tilde{s}_{2r}$ and is symmetric, this proves \eqref{e:dklip}. Moreover, notice that if $x\in \supp\sigma$, $\supp s_{r}(x,\cdot)\subseteq B(x,r)$, and so the integrand of $\tilde{s}_{r}$ is nonzero only when $z\in B(x,r)\cap B(y,r)$, meaning $|x-y|\leq 2r$, and so $\supp\tilde{s}_{r}\subseteq B(x,2r)$, hence $\supp \tilde{d}_{r}\subseteq B(x,4r)$, which proves our claim.

}

\begin{theorem}\cite{DJS85}
	Let $r_{k}=2^{-k}$, and $g\in L^p(\sigma),\ 1<p<\infty$, we have 
	\begin{equation}\label{e:DJS}
	\|g\|_{L^p(\sigma)} \sim \nrm{ \ps{ \sum_{k\in\Z}  |\tilde{D}_{r_{k}} g|^2}^{\frac{1}{2}}}_{L^p(\sigma)}.
	\end{equation}
\end{theorem}

The original result is stated for $p=2$, but this case implies the other cases (see for example the proof of \cite[Corollary 6.1]{tolsa2001littlewood}).

Let $\tilde{D}_{k}:= \tilde{D}_{r_k},\ \tilde{d}_{k}:= \tilde{d}_{r_k}$. By \eqref{e:DJS}, it is clear that to prove \eqref{eq:f smaller than Jf}, it suffices to show that
\begin{equation*}
\nrm{ \ps{ \sum_{k\in\Z}  |\tilde{D}_{{k}} f|^2}^{\frac{1}{2}}}_{L^p(\sigma)} \lesssim \nrm{Jf}_{L^p(\sigma)}.
\end{equation*}	
In fact, we will show a stronger, pointwise inequality which immediately implies the one above.

\begin{lemma}
	Let $x\in \Sigma$, $k\in\mathbb{Z}$, and let $Q\in\D$ be the smallest cube containing $x$ and such that $\supp \tilde{d}_k(x,\cdot)\subset 0.5B_Q$. Then, $\ell(Q)\sim r_k$ and
	\begin{equation}\label{eq:Dkf estimate}
	|\tilde{D}_{k}f(x)| \lesssim \alpha_{f\sigma}(Q) + |f|_{B_Q}\alpha_{\sigma}(Q).                                       \end{equation}
\end{lemma}

The remainder of this section is devoted to the proof of this lemma. 
\subsection{Preliminaries}
Fix $x\in\Sigma,\ k\in\Z,$ and let $Q$ be as above. As noted just above \eqref{e:dklip}, we have $\tilde{d}_k(x,\cdot)\subset B(x,4r_k),$ and so $\ell(Q)\sim r_k$ follows immediately.

We make a few simple reductions.

\begin{remark}
	Without loss of generality we may assume that $\alpha_{\sigma}(Q)\le \varepsilon$ for some small $\varepsilon$. Indeed, if we had $\alpha_{\sigma}(Q)\geq \varepsilon$, then using \eqref{e:dklip} and the fact that $\supp\tilde{d}_{k}(x,\cdot)\subset B_Q$
	\begin{multline*}
	|\tilde{D}_{k}f(x)| = \av{\int \tilde{d}_{k}(x,y) f(y)\ d\sigma(y)} \le \nrm{\tilde{d}_k(x,\cdot)}_{\infty}\int_{B_Q}  |f(y)|\ d\sigma(y)\\
	\lesssim \ell(Q)^{-d} \int_{B_Q}  |f(y)|\ d\sigma(y)\sim |f|_{B_Q}\lesssim_{\varepsilon} |f|_{B_Q}\alpha_{\sigma}(Q),
	\end{multline*}
	and so in this case \eqref{eq:Dkf estimate} holds. From now on we assume $\alpha_{\sigma}(Q)\le \varepsilon$.
\end{remark}

\begin{remark}
	Similarly, without loss of generality we may assume that $L^{f\sigma}_Q\cap 0.5B_Q \neq \varnothing$. If we had $L^{f\sigma}_Q\cap 0.5B_Q = \varnothing$, then $L^{f\sigma}_Q\cap \supp \tilde{d}_r(x,\cdot) = \varnothing$ so that
	\begin{equation*}
	\int \tilde{d}_k(x,y)\ d\cL_Q^{f\sigma}(y) = 0.
	\end{equation*}
	This implies
	\begin{equation*}
	|\tilde{D}_{k}f(x)|=\av{\int \tilde{d}_k(x,y)f(y)\ d\sigma(y)}\lesssim \alpha_{f\sigma}(B),
	\end{equation*}
	and so \eqref{eq:Dkf estimate} is true also in this case.
\end{remark}

Recall that $c^{f\sigma}_Q,\ c^{\sigma}_Q$ are the constants minimizing $\alpha_{f\sigma}(Q),\ \alpha_{\sigma}(Q)$, respectively. Since $\sigma$ is Ahlfors regular and $\alpha_{\sigma}(Q)\le\varepsilon$, we get by \cite[Lemma 3.3]{ATT18}
\begin{equation}\label{eq:csigma nice}
c^{\sigma}_Q\sim 1.
\end{equation}

To show \eqref{eq:Dkf estimate} we begin by using \eqref{e:D1=0} and the triangle inequality:
\begin{multline}\label{eq:Dkf estimate triangle ineq}
	|\tilde{D}_{k}f(x)|
	 = \av{\int_{\Sigma}\tilde{d}_{k}(x,y)f(y)\ d\sigma(y)}
	\stackrel{\eqref{e:D1=0}}{=}\av{\int_{\Sigma}\tilde{d}_{k}(x,y)f(y)\ d\sigma(y) - \frac{c^{f\sigma}_Q}{c^{\sigma}_Q}\int_{\Sigma}\tilde{d}_{k}(x,y)\ d\sigma(y)}\\
	  \leq \av{\int_{\Sigma}\tilde{d}_{k}(x,y)f(y)\ d\sigma(y)  - \int_{L_{Q}^{f\sigma}}\tilde{d}_{k}(x,y)\ d\mathcal{L}_{Q}^{f\sigma}(y) }\\
	  +\av{\int_{L_{Q}^{f\sigma}}\tilde{d}_{k}(x,y)\ d\mathcal{L}_{Q}^{f\sigma}(y) - \frac{c^{f\sigma}_Q}{c^{\sigma}_Q}\int_{L_{Q}^{\sigma}}\tilde{d}_{k}(x,y)\ d\mathcal{L}_{Q}^{\sigma}(y)} \\
	 + \av{\frac{c^{f\sigma}_Q}{c^{\sigma}_Q}}\av{\int_{L_{Q}^{\sigma}}\tilde{d}_{k}(x,y)\ d\mathcal{L}_{Q}^{\sigma}(y) -\int_{\Sigma}\tilde{d}_{k}(x,y)\ d\sigma(y)} =: (I)+(II)+(III). 
\end{multline}
Using the Lipschitz property of $\tilde{d_k}$ \eqref{e:dklip} we immediately get that $(I)\lesssim \alpha_{f\sigma}(Q)$, and that
\begin{equation}\label{eq:estimate of III}
	(III)\lesssim \av{\frac{c^{f\sigma}_Q}{c^{\sigma}_Q}}\alpha_{\sigma}(Q)\stackrel{\eqref{eq:csigma nice}}{\sim}\av{c^{f\sigma}_Q}\alpha_{\sigma}(Q).
\end{equation}
\begin{lemma}
	We have $\av{c^{f\sigma}_Q}\lec |f|_{B_Q}$.
\end{lemma}
\begin{proof}
	Indeed, if we had $\av{c^{f\sigma}_Q}\ge \Lambda |f|_{B_Q}$ for some big $\Lambda>10$, then $\tilde{c}^{f\sigma}_Q=0$ would be a better competitor for a constant minimizing $\alpha_{f\sigma}(Q)$. To see that, note that for any $\varphi\in\lip_1(B_Q)$
	\begin{equation*}
	\av{\int\varphi f\ d\sigma-0} \le {C} \ell(Q)^{d+1} |f|_{B_Q}.
	\end{equation*}
	That is, $F_{B_{Q}}(f\sigma, 0)\le {C} \ell(Q)^{d+1} |f|_{B_Q}.$ On the other hand, taking a positive $\psi\in\lip_1(B_Q)$ such that $\psi(x)= \ell(Q)$ for $x\in 0.7 B_Q$, and using the assumption $L^{f\sigma}_Q\cap 0.5 B_Q\neq\varnothing$ we get
	\begin{multline*}
	\alpha_{f\sigma}(Q)\ell(Q)^{d+1}\gtrsim \av{\int\psi f\ d\sigma-c^{f\sigma}_Q\int_{L^{f\sigma}_Q}\psi \ d\Hd}\ge \av{c^{f\sigma}_Q}\ell(Q) \Hd(0.7B_Q\cap L^{f\sigma}_Q) - \av{\int\psi f\ d\sigma}\\
	\ge \tilde{C}\Lambda |f|_{B_Q}\ell(Q)^{d+1} - {C}\ell(Q)^{d+1} |f|_{B_Q}\ge \frac{\tilde{C}\Lambda}{2}|f|_{B_Q}\ell(Q)^{d+1}> F_{B_{Q}}(f\sigma, 0),
	\end{multline*}
	assuming $\Lambda$ big enough. This contradicts the optimality of $c^{f\sigma}_Q$.
\end{proof}
Using the lemma above and \eqref{eq:estimate of III} we get
\begin{equation*}
(III)\lesssim |f|_{B_Q}\alpha_{\sigma}(Q).
\end{equation*}
Hence, by \eqref{eq:Dkf estimate triangle ineq}, to finish the proof of \eqref{eq:Dkf estimate} it remains to show that
\begin{equation*}
(II) = \av{c^{f\sigma}_Q}\av{\int_{L_{Q}^{f\sigma}}\tilde{d}_{k}(x,y)\ d\Hd(y) - \int_{L_{Q}^{\sigma}}\tilde{d}_{k}(x,y)\ d\Hd(y)} \lesssim \alpha_{f\sigma}(Q) + |f|_{B_Q}\alpha_{\sigma}(Q).
\end{equation*}
This can be seen as an estimate of how far from each other the planes $L^{f\sigma}_Q$ and $L^{\sigma}_Q$ are. 

The inequality above follows immediately from \propref{prop:angles estimate} proven in the next subsection, together with the already established estimate $\av{c^{f\sigma}_Q}\lec |f|_{B_Q}$.

\subsection{Angles between planes approximating \texorpdfstring{$f\sigma$ and $\sigma$}{fsigma and sigma}}
	In the following proposition we do not use uniform rectifiability in any way, and so we state it for a general Ahlfors regular measure $\mu$. Recall that given a ball $B$ we defined $\cP_B^{\mu}=\Hd\mres L^{\mu}_B.$  
	\begin{prop}\label{prop:angles estimate}
		Let $\mu$ be an Ahlfors $d$-regular measure on $\R^n$, and let $f\in L^{1}_{loc}(\mu)$. Let $x\in\supp\mu,\ r>0,\ B=B(x,r)$, and suppose that $L^{f\mu}_B\cap 0.5B\neq\varnothing$. Then,
		\begin{equation}\label{eq:angle estimated with alphas}
		|c_B^{f\mu}| \frac{1}{r^{d+1}}F_{B}(\cP_B^{\mu}, \cP_B^{f\mu})  \lesssim \alpha_{f\mu}(B) + |c_B^{f\mu}|\alpha_{\mu}(B).
		\end{equation}
	\end{prop}

	In the proof of \propref{prop:angles estimate} we will use the following lemma.

	\begin{lemma}
		Let $B=B(x,r)$ and let $L_1,\ L_2$ be two $d$-planes intersecting $0.5B$. Set $\cP_1=\Hr{L_1},\ \cP_2=\Hr{L_2}$. Then,
		\begin{equation}\label{eq:angle controls F distance}
		\frac{1}{r^d}F_{B}(\cP_1, \cP_2)\lesssim \dist_H(L_1\cap B, L_2\cap B).
		\end{equation}
	\end{lemma}
	\begin{proof}
		First, set
		\begin{equation*}
		D=\frac{\dist_H(L_1\cap B, L_2\cap B)}{r}.
		\end{equation*}	
		Note that we always have $F_{B}(\cP_1, \cP_2)\lesssim r^{d+1}$ so that if $D\gtrsim 1,$ then \eqref{eq:angle controls F distance} follows trivially. Hence, without loss of generality we may assume that $D\le \varepsilon$ for some $\varepsilon>0$ to be fixed later.
		
		We claim that if $\varepsilon$ is chosen small enough (depending only on $n,\ d$), then there exists an isometry $A:L_1\to L_2$ such that for $y\in B\cap L_1$ we have $|y-A(y)|\lesssim Dr$. To see that, let $y_1\in L_1\cap B$ be arbitrary. Set $y_2 = \pi_{L_2}(y_1)$. Clearly,
		\begin{equation*}
		|y_1-y_2|\le Dr\le \varepsilon r.
		\end{equation*}		
		Let $v_1,\dots,v_d$ be an orthonormal basis of the linear plane $L_1':=L_1-y_1$. For $i=1,\dots, d$ define
		\begin{equation*}
		w_i := \pi_{L_2}(y_1+v_i) - y_2\in L_2-y_2=:L_2'.
		\end{equation*}
		In fact, since $y_2=\pi_{L_2}(y_1)$, we have $w_i=\pi_{L_2'}(v_i)$. It is easy to see that for all $v\in L'_1$ we have
		\begin{equation*}
		|\pi_{L_2'}(v) - v|\lesssim D|v|.
		\end{equation*}
		Hence, $|w_i - v_i|\lesssim D\le \varepsilon$ and for $i\neq j$
		\begin{equation*}
		|w_i\cdot w_j| = |(w_i-v_i)\cdot (w_j-v_j) + (w_i-v_i)\cdot v_j + v_i\cdot (w_j-v_j)|\lesssim D\le \varepsilon.
		\end{equation*}
		Choosing $\varepsilon$ small enough (depending only on dimensions), we get easily that $\{w_i\}$ is a basis of $L_2'$. Moreover, if $\{\hat{w_i}\}$ is the orthonormal basis of $L_2'$ constructed from $\{w_i\}$ using the Gram-Schmidt process, then it follows from the estimates above that for all $i=1,\dots,d$
		\begin{equation*}
		|\hat{w_i}-v_i|\lesssim D.
		\end{equation*}
		We define the map $A:L_1\to L_2$ as the unique isometry such that $A(y_1)=y_2$ and $A(y_1+v_i)=y_2 + \hat{w_i}$. It follows immediately from basic linear algebra that for $y\in L_1\cap B$ we have $|y-A(y)|\lesssim Dr$.
		
		Now, let $\varphi\in\lip_1(B)$. We have
		\begin{multline*}
		\bigg|\int_{L_1}\varphi(y)\ d\Hd(y) - \int_{L_2}\varphi(y)\ d\Hd(y)\bigg| = \bigg|\int_{L_1}\varphi(y)\ d\Hd(y) - \int_{L_1}\varphi(A(y))\ d\Hd(y)\bigg|\\
		\le \int_{L_1}|\varphi(y)-\varphi(A(y))|\ d\Hd(y)\lesssim \int_{L_1\cap B} Dr\ d\Hd(y)\lesssim Dr^{d+1}. 
		\end{multline*}
		Taking supremum over $\varphi\in\lip_1(B)$ finishes the proof.
	\end{proof}
	
	\begin{proof}[Proof of \propref{prop:angles estimate}]
		For simplicity of notation we will usually omit the subscript $B$, i.e. we will write $L^{\mu}:=L_{B}^{\mu},\ c^{f\mu}:=c^{f\mu}_B$, and so on. 
		
		Without loss of generality we can assume that $c^{f\mu}\ge 0$. Indeed, if that was not the case we could consider $g=-f$. Then the plane and constant $L^{g\mu}=L^{f\mu},\ c^{g\mu} = -c^{f\mu}\ge 0$ are minimizing for $\alpha_{g\mu}(B)$, and we have $\alpha_{g\mu}(B) = \alpha_{f\mu}(B)$. Thus, proving \eqref{eq:angle estimated with alphas} for $g$ is equivalent to proving it for $f$, and $c^{g\mu}\ge 0$.
		
		Note that we always have $F_{B}(\cP^{\mu}, \cP^{f\mu})\lesssim r^{d+1}$ so that if $\alpha_{\mu}(B)\gtrsim 1$, then \eqref{eq:angle estimated with alphas} is trivial. Assume that $\alpha_{\mu}(B)\le \varepsilon$ for some small $\varepsilon>0$ (depending on dimensions and Ahlfors regularity constants), to be fixed later. 
		
		Note that if $\varepsilon$ is small enough, then one can use Ahlfors regularity of $\mu$ to conclude that $L^{\mu}\cap 0.5B\neq\varnothing$ (see for example \cite[Lemma 3.1]{Tol09}). We use this observation, the assumption $L^{f\mu}\cap 0.5B\neq\varnothing$ and \eqref{eq:angle controls F distance} to estimate
		\begin{equation*}
		c^{f\mu}\frac{1}{r^{d+1}}F_{B}(\cP^{\mu}, \cP^{f\mu}) \lesssim c^{f\mu}\frac{\dist_H(L^{\mu}\cap B, L^{f\mu}\cap B)}{r}=:c^{f\mu} D.
		\end{equation*}
		Our aim is to show that
		\begin{equation}\label{eq:ultimate goal}
		c^{f\mu} D \lesssim c^{f\mu}\alpha_{\mu}(B)+\alpha_{f\mu}(B).
		\end{equation}
		
		Let $0<\eta<0.01$ be some dimensional constant. Note that, since $L^{f\mu}\cap 0.5B\neq\varnothing$, the set $L^{f\mu}\cap 0.9 B$ is a $d$-dimensional ball with $\Hd(L^{f\mu}\cap 0.9 B)\sim r^d$. We claim that we can find a $d$-dimensional ball $\sB_0$ contained in  $L^{f\mu}\cap 0.9 B$, of radius $\eta r$ (in particular $r_{\sB_0}\sim_\eta r_B$), and such that 
		\begin{equation}\label{eq:B0 far from Psigma}
		\dist(z,L^{\mu})\ge 10\eta Dr\quad\quad\text{for all $z\in \sB_0$.}
		\end{equation}
		Indeed, if there was no such ball, i.e. if for all $d$-dimensional balls $\sB_0\subset L^{f\mu}\cap 0.9 B$ of radius $\eta r$ there was some $z\in\sB_0$ with $\dist(z,L^{\mu})\le 10\eta Dr$, then it would follow easily from the definition of Hausdorff distance, and from the fact that $L^{\mu}$ and $L^{f\mu}$ are $d$-planes intersecting $0.5B$, that
		\begin{equation*}
		\dist_H(L^{\mu}\cap B, L^{f\mu}\cap B)\lesssim \eta Dr = \eta \dist_H(L^{\mu}\cap B, L^{f\mu}\cap B).
		\end{equation*}
		For $\eta$ small enough, this is a contradiction. We omit the details, which can be readily filled in e.g. using \cite[Lemma 6.4]{azzam2015characterization}.
		
		Consider an open neighbourhood of $\sB_0$ given by
		\begin{equation*}
		U:=\{y\in\R^n\ :\ \dist(y,\sB_0)<\eta D r\},
		\end{equation*}
		and also for $\lambda>0$ set
		\begin{equation*}
		\lambda U:=\{y\in\R^n\ :\ \dist(y,\sB_0)<\lambda\eta D r\}.
		\end{equation*}
		Since $D\le 1$, one should think of $U$ as an $n$-dimensional pancake around $\sB_0$ of thickness $\eta D r$, so that the smaller $D$, the flatter the pancake. Note that by \eqref{eq:B0 far from Psigma} for all $0<\lambda<10$ we have $\lambda U\cap L^{\mu}=\varnothing$, and also $\lambda U\subset B$ because $\sB_0\subset 0.9B$.
		
		Let $\varphi:\R^d\to [0,\eta D r]$ be a function satisfying $\varphi\equiv \eta D r$ in $U$, $\supp\varphi\subset 2U$, and $\lip(\varphi)\le 1$. Clearly, $\varphi\in\lip_1(B)$, and so
		\begin{equation}
		\bigg|\int\varphi f\ d\mu - \int\varphi\ d\cL^{f\mu}\bigg|\le \alpha_{f\mu}(B)r^{d+1}.\label{eq:varphi on fsigma 1}
		\end{equation}
		Furthermore, note that $\varphi\equiv\eta Dr$ on $\sB_0$, so that
		\begin{equation*}
		\int\varphi\ d\cL^{f\mu} = c^{f\mu}\int_{L^{f\mu}}\varphi\ d\Hd \ge c^{f\mu} \eta D r \Hd(\sB_0)= C(d) c^{f\mu} D \eta^{d+1}r^{d+1}.
		\end{equation*}
		Together with \eqref{eq:varphi on fsigma 1} this implies
		\begin{equation}\label{eq:varphi on fsigma 2}
		\int\varphi f\ d\mu \ge C(\eta,d) c^{f\mu} D r^{d+1} - \alpha_{f\mu}(B) r^{d+1}.
		\end{equation}
		Recall that we are trying to prove $c^{f\mu} D \lesssim c^{f\mu}\alpha_{\mu}(B)+\alpha_{f\mu}(B)$. If we had $c^{f\mu}D\le \Lambda \alpha_{f\mu}(B)$ for some $\Lambda=\Lambda(\eta,d)>100$, then there is nothing to prove. So without loss of generality assume that $c^{f\mu}D\ge \Lambda \alpha_{f\mu}(B)$. In that case \eqref{eq:varphi on fsigma 2} gives
		\begin{equation}\label{eq:varphi on fsigma 3}
		\int\varphi f\ d\mu \gtrsim_{\eta} c^{f\mu} D r^{d+1}.
		\end{equation}
		
		Now we define a modified version of $\varphi$. Recall that $\supp\varphi\subset 2U$. For all $y\in\supp\mu\cap 2U$ let $B_y = B(y,\eta D r/5)$. We use the $5r$ covering theorem to extract from $\{B_y\}_{y\in \supp\mu\cap 2U}$ a subfamily of pairwise disjoint balls $\{B_i\}_{i\in I}$ such that $\supp\mu\cap 2U\subset \bigcup_i 5B_i$. Note that $\bigcup_i 10B_i\subset 4U$, and in particular, $\bigcup_i 10B_i\cap L^{\mu}=\varnothing$. Moreover, the balls $10B_i$ have bounded intersection. Thus, we may consider a partition of unity
		\begin{equation*}
		\Psi = \sum_{i\in I}\psi_i,
		\end{equation*}
		such that $\supp\psi_i\subset 10B_i$ for each $i\in I$, $\Psi\equiv 1$ on $\bigcup_i 5B_i$, and $\lip\Psi\lesssim (\eta D r)^{-1}$.
		
		Consider $\Phi = \varphi\Psi$. We have
		\begin{equation*}
		\norm{\nabla\Phi}_{\infty}\le \norm{\nabla\varphi}_{\infty}\norm{\Psi} _{\infty} + \norm{\varphi}_{\infty}\norm{\nabla\Psi} _{\infty}\lesssim 1 + \eta D r (\eta D r)^{-1} = 1.
		\end{equation*}
		Hence, $C\Phi\in\lip_1(B)$ for some $C\sim 1$, so that
		\begin{equation}\label{eq:tildephi on fsigma}
		\bigg|\int\Phi f\ d\mu - \int\Phi\ d\cL^{f\mu}\bigg|\le C^{-1} \alpha_{f\mu}(B)r^{d+1}.
		\end{equation}
		On the other hand, observe that $\Psi\equiv 1$ on $\supp\varphi\cap\supp\mu$. By \eqref{eq:varphi on fsigma 3}
		\begin{equation*}
		\int\Phi f\ d\mu = \int\varphi f\ d\mu\gtrsim_{\eta} c^{f\mu} D r^{d+1}.
		\end{equation*}
		Together with \eqref{eq:tildephi on fsigma} this gives
		\begin{equation}\label{eq:tildephi on Lfsigma}
		\int\Phi\ d\cL^{f\mu}\ge C(\eta)c^{f\mu} Dr^{d+1} - C^{-1}\alpha_{f\mu}(B) r^{d+1}\gtrsim_{\eta}c^{f\mu} Dr^{d+1},
		\end{equation}
		where we used once again the additional assumption $c^{f\mu} D\ge \Lambda\alpha_{f\mu}(B)$ we made along the way (and choosing $\Lambda$ large). 
		
		Now we will show that
		\begin{equation}\label{eq:Phi on Pfsigma{key}}
		\int_{L^{f\mu}}\Phi\ d\Hd \lesssim_{\eta} \alpha_{\mu}(B)r^{d+1}.
		\end{equation}
		Since $\cL^{f\mu} = c^{f\mu}\Hr{L^{f\mu}}$, together with \eqref{eq:tildephi on Lfsigma} this will give $c^{f\mu}D\lesssim_{\eta} c^{f\mu} \alpha_{\mu}(B)$, and so the proof of \eqref{eq:ultimate goal} will be finished.
		
		Recall that $\supp\Phi\subset \supp\Psi\subset \bigcup_i 10B_i$, and that $\norm{\Phi}_{\infty}\le\norm{\varphi}_{\infty}=\eta D r$. Hence,
		\begin{equation*}
		\int_{L^{f\mu}}\Phi\ d\Hd \lesssim_{\eta}Dr \sum_{i\in I}\Hd(L^{f\mu}\cap 10B_i)\lesssim_{\eta} \#I (Dr)^{d+1}.
		\end{equation*}
		To estimate $\#I$ we will use Ahlfors regularity of $\mu$. Recall that $\{B_i\}_{i\in I}$ are pairwise disjoint, they are centered at points from $\supp\mu\cap 2U$, and $r(B_i)=\eta r D/5$. Thus,
		\begin{equation*}
		\#I (rD)^d \sim_\eta \sum_{i\in I}\mu(B_i)= \mu\big(\bigcup_{i\in I} B_i\big).
		\end{equation*}
		On the other hand, since the balls $\{B_i\}$ are centered at points from $2U$, we have $\bigcup_{i\in I} B_i\subset 3U$ and
		\begin{equation*}
		\mu\big(\bigcup_{i\in I} B_i\big)\le \mu(3U).
		\end{equation*}
		To bound $\mu(3U)$ consider $\tilde{\varphi}\in\lip_1(B)$ such that $\tilde{\varphi}\ge 0,\ \tilde{\varphi}\equiv \eta rD$ on $3U$ and $\supp\tilde{\varphi}\subset 4U$. Recalling that $4U\cap L^{\mu}=\varnothing$, we arrive at
		\begin{equation*}
		rD\mu(3U)\lesssim_{\eta} \int\tilde{\varphi}\ d\mu = \bigg|\int\tilde{\varphi}\ d\mu - \int\tilde{\varphi}\ d\cL^{\mu}\bigg| \le \alpha_{\mu}(B)r^{d+1}.
		\end{equation*} 
		Putting all the estimates above together we get \eqref{eq:Phi on Pfsigma{key}}:
		\begin{equation*}
		\int_{L^{f\mu}}\Phi\ d\Hd \lesssim_{\eta} \#I (Dr)^{d+1}\lesssim_{\eta}rD\mu(3U)\lesssim_{\eta}\alpha_{\mu}(B)r^{d+1}.
		\end{equation*}
	\end{proof}

\def\cprime{$'$}


\begin{thebibliography}{DEM18}
	\expandafter\ifx\csname url\endcsname\relax
	\def\url#1{\texttt{#1}}\fi
	\expandafter\ifx\csname doi\endcsname\relax
	\def\doi#1{\burlalt{doi:#1}{http://dx.doi.org/#1}}\fi
	\expandafter\ifx\csname urlprefix\endcsname\relax\def\urlprefix{URL }\fi
	\expandafter\ifx\csname href\endcsname\relax
	\def\href#1#2{#2}\fi
	\expandafter\ifx\csname burlalt\endcsname\relax
	\def\burlalt#1#2{\href{#2}{#1}}\fi
	
	\bibitem[ADT16]{ADT16}
	J.~Azzam, G.~David, and T.~Toro.
	\newblock {Wasserstein} distance and the rectifiability of doubling measures:
	part {I}.
	\newblock {\em Math. Ann.}, 364(1-2):151--224, 2016,
	\burlalt{arXiv:1408.6645}{http://arxiv.org/abs/1408.6645}.
	\newblock \doi{10.1007/s00208-015-1206-z}.
	
	\bibitem[AT15]{azzam2015characterization}
	J.~Azzam and X.~Tolsa.
	\newblock Characterization of $n$-rectifiability in terms of {Jones'} square
	function: Part {II}.
	\newblock {\em Geom. Funct. Anal.}, 25(5):1371--1412, 2015,
	\burlalt{arXiv:1501.01572}{http://arxiv.org/abs/1501.01572}.
	\newblock \doi{10.1007/s00039-015-0334-7}.
	
	\bibitem[ATT18]{ATT18}
	J.~{Azzam}, X.~{Tolsa}, and T.~{Toro}.
	\newblock {Characterization of rectifiable measures in terms of
		$\alpha$-numbers}.
	\newblock {\em Preprint, to appear in Trans. Amer. Math. Soc.}, 2018,
	\burlalt{arXiv:1808.07661}{http://arxiv.org/abs/1808.07661}.
	\newblock \doi{10.1090/tran/8170}.
	
	\bibitem[Chr90]{christ1990tb}
	M.~Christ.
	\newblock A {$T(b)$} theorem with remarks on analytic capacity and the {Cauchy}
	integral.
	\newblock {\em Colloq. Math.}, 60:601--628, 1990.
	\newblock \doi{10.4064/cm-60-61-2-601-628}.
	
	\bibitem[D{\k{a}}b19]{Dab19}
	D.~D{\k{a}}browski.
	\newblock Necessary condition for rectifiability involving {Wasserstein}
	distance {$W_2$}.
	\newblock {\em Preprint, to appear in Int. Math. Res. Not. IMRN}, 2019,
	\burlalt{arXiv:1904.11000}{http://arxiv.org/abs/1904.11000}.
	\newblock \doi{10.1093/imrn/rnaa012}.
	
	\bibitem[Dav88]{david1988morceaux}
	G.~David.
	\newblock Morceaux de graphes lipschitziens et int{\'e}grales singulieres sur
	une surface.
	\newblock {\em Rev. Mat. Iberoam.}, 4(1):73--114, 1988.
	\newblock \doi{10.4171/RMI/64}.
	
	\bibitem[DEM18]{DEM18}
	G.~David, M.~Engelstein, and S.~Mayboroda.
	\newblock Square functions, non-tangential limits and harmonic measure in
	co-dimensions larger than one.
	\newblock {\em Preprint, to appear in Duke Math. J.}, 2018,
	\burlalt{arXiv:1808.08882}{http://arxiv.org/abs/1808.08882}.
	
	\bibitem[DJS85]{DJS85}
	G.~David, J.-L. Journ{\'e}, and S.~Semmes.
	\newblock {Op{\'e}rateurs de Calder{\'o}n-Zygmund, fonctions
		para-accr{\'e}tives et interpolation}.
	\newblock {\em Rev. Mat. Iberoam.}, 1(4):1--56, 1985.
	\newblock \doi{10.4171/RMI/17}.
	
	\bibitem[DM20]{david2020harmonic}
	G.~David and S.~Mayboroda.
	\newblock Harmonic measure is absolutely continuous with respect to the
	Hausdorff measure on all low-dimensional uniformly rectifiable sets.
	\newblock {\em arXiv preprint}, 2020,
	\burlalt{arXiv:2006.14661}{http://arxiv.org/abs/2006.14661}.
	
	\bibitem[DS91]{DS}
	G.~David and S.~Semmes.
	\newblock Singular integrals and rectifiable sets in {$\mathbb{R}^n$}:
	Au-del\`{a} des graphes lipschitziens.
	\newblock {\em Ast{\'e}risque}, 193, 1991.
	\newblock \doi{10.24033/ast.68}.
	
	\bibitem[Fen20]{feneuil2020absolute}
	J.~Feneuil.
	\newblock Absolute continuity of the harmonic measure on low dimensional
	rectifiable sets.
	\newblock {\em arXiv preprint}, 2020,
	\burlalt{arXiv:2006.03118}{http://arxiv.org/abs/2006.03118}.
	
	\bibitem[Gra14]{grafakos2014classical}
	L.~Grafakos.
	\newblock {\em {Classical Fourier Analysis}}, volume 249 of {\em Grad. Texts in
		Math.}
	\newblock Springer-Verlag New York, 3 edition, 2014.
	\newblock \doi{10.1007/978-1-4939-1194-3}.
	
	\bibitem[HT14]{hytonen2014almost}
	T.~Hyt{\"o}nen and O.~Tapiola.
	\newblock Almost Lipschitz-continuous wavelets in metric spaces via a new
	randomization of dyadic cubes.
	\newblock {\em J. Approx. Theory}, 185:12--30, 2014,
	\burlalt{arXiv:1310.2047}{http://arxiv.org/abs/1310.2047}.
	\newblock \doi{10.1016/j.jat.2014.05.017}.
	
	\bibitem[Mat95]{mattila1999geometry}
	P.~Mattila.
	\newblock {\em Geometry of sets and measures in {Euclidean} spaces: fractals
		and rectifiability}, volume~44 of {\em Cambridge Stud. Adv. Math.}
	\newblock Cambridge Univ. Press, 1995.
	\newblock \doi{10.1017/CBO9780511623813}.
	
	\bibitem[Orp18]{Orp18}
	T.~Orponen.
	\newblock Absolute continuity and $\alpha$-numbers on the real line.
	\newblock {\em Anal. PDE}, 12(4):969--996, 2018,
	\burlalt{arXiv:1703.02935}{http://arxiv.org/abs/1703.02935}.
	\newblock \doi{10.2140/apde.2019.12.969}.
	
	\bibitem[Tol01]{tolsa2001littlewood}
	X.~Tolsa.
	\newblock {Littlewood--Paley theory and the $T(1)$ theorem with non-doubling
		measures}.
	\newblock {\em Adv. Math.}, 164(1):57--116, 2001,
	\burlalt{arXiv:math/0006039}{http://arxiv.org/abs/math/0006039}.
	\newblock \doi{10.1006/aima.2001.2011}.
	
	\bibitem[Tol08]{Tol08}
	X.~Tolsa.
	\newblock Principal values for {R}iesz transforms and rectifiability.
	\newblock {\em J. Funct. Anal.}, 254(7):1811--1863, 2008,
	\burlalt{arXiv:0708.0109}{http://arxiv.org/abs/0708.0109}.
	\newblock \doi{j.jfa.2007.07.020}.
	
	\bibitem[Tol09]{Tol09}
	X.~Tolsa.
	\newblock Uniform rectifiability, {Calder{\'o}n}-{Zygmund} operators with odd
	kernel, and quasiorthogonality.
	\newblock {\em Proc. Lond. Math. Soc. (3)}, 98(2):393--426, 2009,
	\burlalt{arXiv:0805.1053}{http://arxiv.org/abs/0805.1053}.
	\newblock \doi{10.1112/plms/pdn035}.
	
	\bibitem[Tol12]{Tol12}
	X.~Tolsa.
	\newblock Mass transport and uniform rectifiability.
	\newblock {\em Geom. Funct. Anal.}, 22(2):478--527, 2012,
	\burlalt{arXiv:1103.1543}{http://arxiv.org/abs/1103.1543}.
	\newblock \doi{10.1007/s00039-012-0160-0}.
	
	\bibitem[Tol14]{tolsa2014analytic}
	X.~Tolsa.
	\newblock {\em Analytic capacity, the Cauchy transform, and non-homogeneous
		{Calder{\'o}n}-{Zygmund} theory}, volume 307 of {\em Progress in
		Mathematics}.
	\newblock Birkhäuser, 2014.
	\newblock \doi{10.1007/978-3-319-00596-6}.
	
	\bibitem[Tol17]{tolsa2017rectifiable}
	X.~Tolsa.
	\newblock Rectifiable measures, square functions involving densities, and the
	{Cauchy} transform.
	\newblock {\em Mem. Amer. Math. Soc.}, 245(1158), 2017,
	\burlalt{arXiv:1408.6979}{http://arxiv.org/abs/1408.6979}.
	\newblock \doi{10.1090/memo/1158}.
	
\end{thebibliography}
\end{document}